\documentclass[11pt]{amsart}
\usepackage{
 amsmath, 
 amsxtra, 
 amsthm, 
 amssymb, 
 etex, 
 mathrsfs, 
 mathtools, 
 tikz-cd, 
 bbm,
 xr,
 comment,
 enumitem}
\usepackage[toc]{appendix} 
\usepackage[all]{xy}
\usepackage{hyperref}
\usepackage{url}
\usepackage{tipa}
\usepackage{dsfont}
\usepackage{ textcomp }
\usepackage{stmaryrd} 
\usepackage{amssymb}
\usepackage{mathabx} 
\usepackage{amsmath} 
\usepackage{amscd}
\usepackage{amsbsy}
\usepackage{comment, enumerate}
\usepackage{color}
\usepackage{mathtools,caption}
\usepackage{tikz-cd}
\usepackage{longtable}
\usepackage[utf8]{inputenc}
\usepackage[OT2,T1]{fontenc}
\usepackage{changepage}
\usepackage{commath,enumerate}
\usetikzlibrary{backgrounds}

\DeclareMathOperator{\End}{End}

\DeclareMathOperator{\Gal}{Gal}

\DeclareMathOperator{\ord}{ord}

\DeclareMathOperator{\GL}{GL}

\newtheorem{theorem}{Theorem}[section]
\newtheorem*{theorem*}{Theorem}
\newtheorem{lemma}[theorem]{Lemma}

\newtheorem{proposition}[theorem]{Proposition}
\newtheorem{corollary}[theorem]{Corollary}

\newtheorem{defn}[theorem]{Definition}

\numberwithin{equation}{section}
\newtheorem{lthm}{Theorem} 

\theoremstyle{remark}
\newtheorem{remark}[theorem]{Remark}
\newtheorem{example}[theorem]{Example}

\newcommand{\EE}{\mathbb{E}}
\newcommand{\R}{\mathbb{R}}
\newcommand\EatDot[1]{}

\newcommand{\cK}{\mathcal{K}}

\newcommand{\cG}{{\mathcal{G}}}

\newcommand{\cT}{{\mathcal{T}}}

\newcommand{\CC}{\mathbb{C}}

\newcommand{\cX}{\mathcal{X}}
\newcommand{\QQ}{\mathbb{Q}}
\newcommand{\ZZ}{\mathbb{Z}}
\newcommand{\FF}{\mathbb{F}}
\newcommand{\Qp}{\mathbb{Q}_p}
\newcommand{\Zp}{\mathbb{Z}_p}

\definecolor{Green}{rgb}{0.0, 0.5, 0.0}

\newcommand{\cO}{\mathcal{O}}

\newcommand{\Q}{\mathbb{Q}}

\newcommand{\Z}{\mathbb{Z}}

\newcommand{\Fl}{\mathbb{F}_l}
\newcommand{\Deck}{\mathrm{Deck}}

\newcommand{\fL}{\mathfrak{L}}
\renewcommand{\r}{\mathfrak r}
\renewcommand{\s}{\mathfrak s}
\renewcommand{\t}{\mathfrak t}
\newcommand{\cc}{\mathfrak c}
\renewcommand{\ss}{\mathrm{ss}}
\newcommand{\cY}{\mathcal{Y}}

  \DeclareFontFamily{U}{wncy}{}
  \DeclareFontShape{U}{wncy}{m}{n}{<->wncyr10}{}
  \DeclareSymbolFont{mcy}{U}{wncy}{m}{n}
  \DeclareMathSymbol{\sha}{\mathord}{mcy}{"58}
  \DeclareMathSymbol{\zhe}{\mathord}{mcy}{"11}

\mathcode`l="8000
\begingroup
\makeatletter
\lccode`\~=`\l
\DeclareMathSymbol{\lsb@l}{\mathalpha}{letters}{`l}
\lowercase{\gdef~{\ifnum\the\mathgroup=\m@ne \ell \else \lsb@l \fi}}%
\endgroup

\makeatletter
\newcommand{\mylabel}[2]{#2\def\@currentlabel{#2}\label{#1}}
\makeatother
  
\title{On towers of isogeny graphs with full level structures}

\makeatletter
\let\@wraptoccontribs\wraptoccontribs
\makeatother

\usepackage{tikz,xcolor,hyperref}

\author[A. Lei]{Antonio Lei}
\address[Lei]{Department of Mathematics and Statistics\\University of Ottawa\\
150 Louis-Pasteur Pvt\\
Ottawa, ON\\
Canada K1N 6N5}
\email{antonio.lei@uottawa.ca}

\author[K. Müller]{Katharina Müller}
\address[Müller]{Institut für Theoretische Informatik, Mathematik und Operations Research, Universität der Bundeswehr München, Werner-Heisenberg-Weg 39, 85577 Neubiberg, Germany}
\email{katharina.mueller@unibw.de}

\subjclass[2020]{Primary:  05C25, 11G20 Secondary: 11R23, 14G17, 14K02}
\keywords{Isogeny graphs with level structure, Iwasawa theory for graphs, oriented supersingular elliptic curves}

\begin{document}
\begin{abstract}
    Let $p$ and $l$ be distinct prime numbers, let $q$ be a power of a prime number $r$ that is distinct from $p$ and $l$, and let $M$ be a positive integer coprime to $ql$. 
     We define the directed graph $X_l^q(M)$  whose vertices are given by isomorphism classes of  elliptic curves over the finite field of $q$ elements equipped with a  level $M$ structure. The edges of $X_l^q(M)$ are given by $l$-isogenies. Fix a positive integer $N$ and write $M=p^nN$. We are interested in when the connected components of $X_l^q(p^nN)$ give rise to a tower of Galois coverings as $n$ varies. We show that only in the supersingular case we do get a tower of Galois coverings. We also study similar towers of  isogeny graphs given by oriented supersingular curves, as introduced by Col\`o--Kohel, enhanced with a level structure. 
\end{abstract}
\maketitle
\section{Introduction}
Let $p$ and $l$ be distinct prime numbers. Let $N$ be a positive integer that is coprime to $pl$. Let $G_N^n$ be the directed graph whose vertices are isomorphism classes $(E,P)$, where $E$ is an ordinary elliptic curve defined over a fixed finite field of characteristic $p$ and $P$ is a point of order $p^nN$ on $E$, and the edges of $G_N^n$ are given by $l$-isogenies. In \cite{LM1}, we studied the tower of graph coverings $$G_N^0\leftarrow G_N^1\leftarrow G_N^2\leftarrow \cdots\leftarrow G_N^n\leftarrow\cdots.$$ 
Note that the degree of the covering $G_N^{n+1}/G_N^n$ is $p$ thanks to the ordinarity condition. We studied the undirected graphs $\tilde G_N^n$ obtained from $G_N^n$ by forgetting the directions of the edges of $G_N^n$. We showed that if $E$ is an elliptic curve representing a non-isolated vertex of $G_1^0$ and $\tilde{\cG}_N^m$ is a connected component of $\cG_N^m$ containing a vertex arising from $E$, there exists an integer $m_0$ such that
\[
\tilde\cG_N^{m_0}\leftarrow \tilde\cG_N^{m_0+1}\leftarrow \tilde\cG_N^{m_0+2}\leftarrow \cdots\leftarrow \tilde\cG_N^{m_0+n}\leftarrow\cdots
\]
is an abelian $p$-tower in the sense of Vallières and McGown-Vallières \cite{vallieres,vallieres2,vallieres3}, i.e. the covering $\tilde\cG_N^{m_0+r}/\tilde\cG_N^{m_0}$ is Galois, whose Galois group is isomorphic to the cyclic group $\ZZ/p^n\ZZ$ for all $n\ge0$. Such towers exhibit properties that resemble $\Zp$-towers of number fields studied in classical Iwasawa theory \cite{iwasawa69,iwasawa73}. For example, similar to the class numbers of number fields inside a $\Zp$-tower, the number of spanning trees of $\tilde\cG_N^{m_0+n}$ can be described explicitly for $n$ sufficiently large (see \cite{vallieres3,leivallieres}).

In this article, we study towers of isogeny graphs that result in non-commutative Galois coverings under appropriate hypotheses. Throughout, we fix a finite field $\FF_q$, where $q$ is a power of a prime number $r$ that is coprime to $pl$. {We say that two elliptic curves defined over $\FF_q$ are equivalent if they are isomorphic over $\overline{\FF_q}$.} We fix a set of representatives of the equivalence classes of elliptic curves defined over $\mathbb{F}_{q}$, which we denote by $S$. 

Given two elliptic curves $E,E'\in S$, an isogeny $\phi:E\to E'$ is a non-constant morphism of varieties that sends the point at infinity of $E$ to the point at infinity of $E'$. Recall that all isogenies are group homomorphisms. We say that $\phi$ is an $\ell$-isogeny if its degree is equal to $l$. Since $l\nmid q$, this is equivalent to $\ker(\phi)$ containing $l$ elements in $E(\overline{\FF_q})$.

\begin{defn}\label{def:intro}
    Let $l,r$ be distinct prime numbers, $q$ a power of $r$ and $M$ a positive integer coprime to $ql$. We define the directed graph $X_l^q(M)$ whose vertices are given by triples $(E,Q_1,Q_2)$, where $E\in S$ and $\{Q_1,Q_2\}$ is a basis of $E[M]\cong (\ZZ/M\ZZ)^2$. The edges of $X_l^q(M)$ are given by $l$-isogenies, i.e. the set of edges from $(E,Q_1,Q_2)$ to $(E',Q_1',Q_2')$ is the set of $l$-isogenies $\phi\colon E\to E'$ such that $\phi(Q_1)=Q_1'$ and $\phi(Q_2)=Q_2'$.
\end{defn}
Note that we allow isogenies and $M$-torsion points to be defined over $\overline{\mathbb{F}_{q}}$, while the elliptic curves are defined over $\mathbb{F}_{q}$. We fix one representative $E$ for each class in $S$. In particular, the vertices $(E,Q_1,Q_2)$ and $(E,-Q_1,-Q_2)$ are considered to be distinct if $M>2$. Furthermore, $\phi$ and $-\phi$ give rise to two distinct edges.  If $\phi$ induces an edge from $(E,Q_1,Q_2)$ to $(E',Q'_1,Q'_2)$, then $-\phi$ induces an edge from $(E,Q_1,Q_2)$ to $(E,-Q'_1,-Q'_2)$.

The basis $\{Q_1,Q_2\}$ is called a $\Gamma(M)$-level structure on the elliptic curve $E$. Fix a positive integer $N$ that is coprime to $pql$. We write $M=p^nN$ for some integer $n\ge0$.
 For $i\in\{1,2\}$, {there exist unique $R_i\in E[N]$ and $S_i\in E[p^n]$ such that $Q_i=R_i+S_i$. Furthermore, $\{R_1,R_2\}$ is a $\Gamma(N)$-level structure and $\{S_1,S_2\}$ is a $\Gamma(p^n)$-level structure on $E$}. We shall denote a vertex of $X_l^q(p^nN)$ by $(E,R_1,R_2,P,Q)$ with $R_1,R_2\in E[N]$ and $P,Q\in E[p^n]$ such that $\{R_1+ P, R_2+ Q\}$ is a  $\Gamma(p^nN)$-level  structure on $E$.  When $n=0$, we can discard $P$ and $Q$, and simply write $(E,R_1,R_2)$ for a vertex in $X_l^q(N)$.

The map $(E,R_1,R_2,P,Q)\mapsto(E,R_1,R_2,pP,pQ)$ induces a graph covering $X_l^q(p^{n+1}N)/ X_l^q(p^{n}N)$ (see Corollary~\ref{cor:cover}).
We are interested in when the connected components of $X_l^q(p^nN)$ give rise to a tower of Galois coverings as $n$ varies. 
Let $X$ be a connected component of $X_l^q(N)$. For $n\ge1$, let $X_n$ be the pre-image of $X$ under the covering $X_l^q(p^nN)/ X_l^q(N)$.
Note that if $v_1,v_2\in V(X_n)$, the elliptic curves that give rise to $v_1$ and $v_2$ are isogenous. In particular, {either both curves are ordinary or both are supersingular. In other words, the ordinary-supersingular dichotomy is constant for all vertices of $X$}. Therefore, it makes sense to refer to a connected component $X$ as  \textbf{\textit{ordinary}} or \textbf{\textit{supersingular}}. It turns out that the coverings $X_n/X$ exhibit different properties depending on {this dichotomy}. In particular, we prove: 

\begin{lthm}[{Corollary~\ref{cor:ord-not-Galois}}, {Theorem~\ref{thm:abelian}}]\label{thmA}
Let $X$ be an ordinary connected component of $X_l^q(N)$. For $n\ge1$, let $X_n$ be the pre-image of $X$ under the covering $X_l^q(p^nN)/ X_l^q(N)$.
\begin{itemize}
    \item[(i)] The covering $X_{n}/X$ is not Galois when $n$ is sufficiently large.
    \item[(ii)] Let $X_0'=X$. For each $n\ge1$, let $X_n'$ be a connected component of the pre-image of $X_{n-1}'$ in $X_l^q(p^nN)$. Then $X_n'/X$ is a Galois covering. Furthermore, $\displaystyle\varprojlim_n \Gal(X_n/X)$ is an abelian $p$-adic Lie group of dimension $1$ or $2$.
\end{itemize}
\end{lthm}
The reason why $X_n/X$ is not Galois is that the number of connected components of $X_n$ grows rapidly in $n$; see Theorem~\ref{thm:ordinary-components} for a precise description.
In the supersingular case, we show:

\begin{lthm}[{Corollaries~\ref{cor:density}, \ref{cor:connected}, \ref{cor:galois-supersingular} and Remarks~\ref{cyclic-for-large-N}},\ref{rem:subgroup}]\label{thmB}
 Suppose that $p\ne 2$. Let $X^\ss$ be a supersingular connected component of $X_l^q(N)$. For $n\ge1$, let $X_n$ be the pre-image of $X^\ss$ under the covering $X_l^q(p^nN)/X_l^q(N)$.
 \begin{itemize}
     \item[(i)] The covering $X^\ss_{n}/ X^\ss$ is Galois for all $n$ if and only if $(\ZZ/Np^2\ZZ)^\times=\langle \ell\rangle$. When this occurs, the Galois group of the covering $X_{n}^\ss/ X^\ss$ is isomorphic to $\GL_2(\ZZ/p^n\ZZ)$.
     \item[(ii)] There exists an integer $m_0$ such that the number of connected components in $X_n^\ss$ stabilizes. Let $Z_{m_0}$ be a connected component of $X_{m_0}^\ss$. For $n\ge  m_0$, let $Z_n$ be the pre-image of $Z_{m_0}$ in $X_n^\ss$. Then $Z_n/Z_{m_0}$ is a Galois covering. Furthermore, $\displaystyle\varprojlim_{n\ge m_0} \Gal(Z_n/Z_{m_0})$ is an open subgroup of $\GL_2(\Zp)$. We can take $m_0=0$ if and only if $[(\ZZ/Np^2\ZZ)^\times:\langle \ell\rangle]=[(\ZZ/N\ZZ)^\times :\langle \ell\rangle]$, in which case the inverse limit of Galois groups is exactly $\GL_2(\Zp)$. In particular, the density of such $l$ is positive.
 \end{itemize}
\end{lthm}

In particular, in the context of Theorem~\ref{thmB},
\[
 X_1\leftarrow X_2\leftarrow\cdots\leftarrow X_n\leftarrow\cdots\]
gives an explicit example that fits into the framework of non-commutative Iwasawa of graph coverings developed in \cite{KM}. We remark that Roda \cite{thesis-roda} proved a simple necessary and sufficient condition for the connectivity of $X_l^q(p^nN)^{\ss}$. We shall review this result in the main body of the article (see Theorem~\ref{thm:number-connected-comp}). This result is utilized to prove the non-triviality of the density of primes satisfying Theorem~\ref{thmB}.

In non-commutative Iwasawa theory of elliptic curves studied in \cite{CFKSV}, the authors conjectured that the Pontryagin dual of the $p$-primary Selmer group of a $p$-ordinary elliptic curve over a $p$-adic Lie extension $\cK_\infty$ of a number field $\cK_0$ that contains the cyclotomic $\Zp$-extension should satisfy the so-called $\mathfrak{M}_H(G)$-property. Here, $G=\Gal(\cK_\infty/\cK_0)$ and $H$ is a subgroup of $G$ such that $\cK_\infty^H$ is the cyclotomic $\Zp$-extension of $\cK_0$. The analogue of the $\mathfrak{M}_H(G)$-property  in the context of graph coverings has recently been studied in \cite{KM}. It is therefore natural to seek an appropriate quotient inside the tower given by Theorem~\ref{thmB} that would play the role of the cyclotomic $\Zp$-extension of a number field. 

Indeed, the graph  $X_l^q(p^nN)$ admits a natural  quotient $Y_l^q(p^nN)$  obtained by the map $$(E,R_2,R_2,P,Q)\mapsto (E,R_1,R_2,\langle P,Q\rangle),$$ where $\langle-,-\rangle$ is the Weil pairing (see Definition~\ref{def-zp-graph}). We show that these graphs give rise to an abelian $p$-tower, similar to the isogeny graphs studied in \cite{LM1}. In particular, we prove:

  \begin{lthm}[{Corollary~\ref{cor:Zp-tower}}]\label{thmC}
      There exists an integer  $m_0$  such that if $Y_{m_0}$ is a connected component of $Y_l^q(p^{m_0}N)$ and $Y_{n+m_0}$ is the pre-image of $Y_{m_0}$ in $Y_l^q(p^{n+m_0}N)$, then $Y_{n+m_0}/Y_{m_0}$ is Galois, with Galois group isomorphic to $\ZZ/p^n\ZZ$ for all $n\ge0$.
  \end{lthm}
  Note that unlike Theorems~\ref{thmA} and \ref{thmB}, where a divergence emerges between the ordinary and supersingular cases, Theorem~\ref{thmC} is independent of whether the chosen connected component is ordinary or supersingular.
  
In \S\ref{S:orientation}, we study isogeny graphs that arise from \textit{oriented elliptic curves}, which were first introduced in \cite{colokohel}.  An orientation of an elliptic curve $E$ is given by an embedding $K\hookrightarrow\End(E)\otimes\QQ$, where $K$ is an imaginary quadratic field. The structure of these graphs (in the supersingular case) has been studied in \cite[\S4]{onuki}. More recently, the authors of \cite{arpin-et-all,arpin-win} described these graphs as volcano graphs, similar to those studied in \cite{kohel}. 

Let $M\ge1$ be an integer that is coprime to $ql$. We define the isogeny graphs of oriented elliptic curves equipped with a $\Gamma(M)$-level structure, which we denote by $\cX_l^q(M)$ (see Definition~\ref{def:oriented-graph-level}). As in the case of $X_l^q(M)$, each connected component of $\cX_l^q(M)$ is either ordinary or supersingular. Furthermore, an imaginary quadratic field $K$ is attached to each connected component, depending on the choice of orientation. We describe the connected components of these graphs explicitly. In particular, we prove: 

\begin{lthm}[{Theorems~\ref{thm:volcano}, \ref{thm:volcano2} and \ref{thm:volcano3}}]
       Assume that  $M>2$. Let $\cY_M$ be a connected component of $\mathcal{X}_l^q(M)$ and let $K$ be the imaginary quadratic field attached to $\cY_N$. When $\cY_M$ is supersingular, we have:
    \begin{itemize}
        \item If $l$ splits in $K$, $\cY_M$ is the double intertwinement of an infinite tectonic volcano graph.
        \item If $l$ ramifies in $K$,  $\cY_M$ is the double intertwinement of an infinite  volcano graph with a connected crater.
        \item If $l$ is inert in $K$,  $\cY_M$ is the double intertwinement of an infinite volcano graph whose crater is disconnected.
    \end{itemize}
    When $\cY_M$ is ordinary, we have:
    \begin{itemize}
        \item If $l$ splits in $K$, $\cY_M$ is the double intertwinement of a finite tectonic volcano graph.
        \item If $l$ ramifies in $K$,  $\cY_M$ is the double intertwinement of a  finite volcano graph with a connected crater.
        \item $l$ is inert in $K$,  $\cY_M$ is the double intertwinement of a  finite volcano graph whose crater is disconnected.
    \end{itemize}
    \end{lthm}
The reader is referred to Definitions~\ref{def:volcano}, \ref{def:intertwine} and \ref{def:crater} where the concepts of "double intertwinement", "volcano graphs", "tectonic volcano graphs", and "craters" are introduced.

The final result of the present article is an analogue of Theorem~\ref{thmA} for isogeny graphs of oriented supersingular elliptic curves with level structures:

\begin{lthm}[{Corollary~\ref{cor:ss-not-Galois}}, {Theorem~\ref{thm:ss-abelian}}]\label{thmE}
Let $\cX$ be an connected component of $\cX_l^q(N)$. Suppose that $\cX$ is supersingular and let $K$ be the imaginary quadratic field attached to $\cX$. For $n\ge1$, let $\cX_n$ be the pre-image of $\cX$ under the covering $\cX_l^q(p^nN)/ \cX_l^q(N)$.
\begin{itemize}
    \item[(i)] The covering $\cX_{n}/\cX$ is not Galois when $n$ is sufficiently large.
    \item[(ii)] Let $\cX_0'=\cX$. For each $n\ge1$, let $\cX_n'$ be a connected component of the pre-image of $\cX_{n-1}'$ in $\cX_l^q(p^nN)$. Then $\cX_n'/\cX$ is a Galois covering. Furthermore, $\displaystyle\varprojlim_n \Gal(\cX_n/\cX)$ is an abelian $p$-adic Lie group of dimension $1$ or $2$.
\end{itemize}
\end{lthm}


We conclude this introduction by highlighting the growing interest in isogeny graphs of elliptic curves with level structures in recent years. One such example is Goren--Kassaei's work \cite{gorenkassaei} that delved into the dynamics of Hecke operators on modular curves utilizing such graphs with $\Gamma_1(N)$-level structure. In a different vein, Arpin \cite{arpin}  investigated the implications of isogeny graphs for equivalence classes of supersingular elliptic curves with $\Gamma_0(N)$-level structures in the realm of isogeny-based cryptography. During the preparation of the present article, we learned about the recent work of Codogni--Lido \cite{codogni-lido}, where they considered more general level structures. Analogously to Theorem~\ref{thm:number-connected-comp}, they studied the number of connected components in these graphs. In addition, they analyzed the eigenvalues of the associated adjacency matrices, establishing connections between these graphs and modular forms (the special case of $\Gamma_0(N)$-level structure was also studied in \cite{sugiyama,LM-zeta}).

\subsection*{Acknowledgement}
We thank Pete Clark, Sören Kleine, and Daniel Vallières for helpful exchanges regarding the content of the present article. We thank the anonymous referees for their helpful comments and suggestions on earlier versions of the article, which led to many improvements. The research of AL is supported by the NSERC Discovery Grants Program RGPIN-2020-04259 and RGPAS-2020-00096. While preparing this article, KM was a postdoctoral fellow at Universit\'e Laval. During that time, her research was supported by the grants mentioned above.

\section{Voltage assignment and graph covering}
In this section, we show that the isogeny graphs with level structures introduced in Definition~\ref{def:intro} can be naturally realized as voltage graphs. Our discussion can be regarded as a natural extension of \cite[Appendix A]{LM1} to the non-commutative setting. 

\subsection{Generalities of voltage graphs}
Let us first recall the definition of a voltage assignment.
\begin{defn}Let $X$ be a directed graph. \label{def:voltage}
    \item[i)] The set of vertices and the set of edges of $X$ are denoted by $V(X)$  and $\EE(X)$, respectively.
    
    \item[ii)] Let  $(G,\cdot)$ a group. A  function $\alpha:\EE(X)\rightarrow G$ is called a \textbf{$G$-valued voltage assignment} on $X$.

    \item[iii)] To each voltage assignment $\alpha$ given as in ii), we define the \textbf{derived graph} $X(G,\alpha)$ to be the graph whose vertices and edges are given by $V(X)\times G$ and $\EE(X)\times G$, respectively. If  $(e,\sigma)\in \EE(X)\times G$, it links $(s,\sigma)$ to $(t,\sigma\cdot\alpha(e))$,  where $e$ is an edge in $X$ that links $s$ to $t$.

    \item[iv)] A graph arises from a voltage assignment is called a \textbf{voltage graph}. 
    
\end{defn}

Let $G$, $X$ and $\alpha$ be given as in Definition~\ref{def:voltage}. There is a natural graph covering $X(G,\alpha)/ X$ given by
\[   V(X(G,\alpha))\ni(v,\sigma)\mapsto v,\quad  \EE(X(G,\alpha))\ni(e,\sigma)\mapsto e.
\]
There is a natural left action by $G$ on $X(G,\alpha)$, given by $g\cdot (v,\sigma)=(v,g\cdot \sigma)$ and $g\cdot(e,\sigma)=(e,g\cdot \sigma)$ for $g,\sigma\in G$, $v\in V(X)$ and $e\in \EE(X)$.

\begin{lemma}\label{lem:intermediate}
    Let $X$, $G$ and $\alpha$ be as in Definition~\ref{def:voltage}. Let $H$ be a normal subgroup of $G$. We write $\tilde\alpha$ for the composition of $\alpha$ with the projection map $G\to G/H$. Then $X(G/H,\tilde\alpha)$ is an intermediate covering of $X(G,\alpha)/X$.
    \end{lemma}
    \begin{proof}
        Let $\pi_1:X(G,\alpha)\to X(G/H,\tilde\alpha)$ be the map given by
    \[   V(X(G,\alpha))\ni(v,\sigma)\mapsto (v,\sigma H),\quad  \EE(X(G,\alpha))\ni(e,\sigma)\mapsto (e,\sigma H).\]
    Let $\pi_2:X(G/H,\tilde\alpha)\to X$ and $\pi_3:X(G,\alpha)\to X$ be the natural coverings. 
    It can be verified by direct calculations that $\pi_1$ is a graph covering and that $\pi_3=\pi_2\circ\pi_1$. Hence, the lemma follows.
    \end{proof}

\begin{defn}
\label{def:connected}Let $X$ be a directed graph.
\item[i)] We write $X'$ for the undirected graph obtained from $X$ by ignoring the directions of the edges of $X$. The natural map from $X$ to $X'$ is called the \textbf{forgetful map}.
   \item[ii)]   We say that $X$ is \textbf{connected} if $X'$ is connected. 
\item[iii)] A \textbf{connected component} of $X$ is the pre-image of a connected component of $X'$ in $X$ under the forgetful map.
\item[iv)] If $Y/X$ is a covering of directed graph with the projection map $\pi:Y\to X$, the group of \textbf{deck transformations} of $Y/X$, denoted by $\Deck(Y/X)$, is the group of graph automorphisms $\sigma:Y\to Y$ such that $\pi\circ\sigma=\pi$. 
\item[v)]     We call a covering $Y/X$ of directed graphs \textbf{$d$-sheeted} if $d$ is a positive integer such that each element of $V(X)$ has $d$ pre-images in $V(Y)$.
\item[vi)] We say that a covering $Y/X$ of directed graphs it is \textbf{Galois}  {if $Y/X$ is a $d$-sheeted covering and $\vert \textup{Deck}(Y/X)\vert =d$.}
\end{defn}
The notion of connectedness here is sometimes referred to as "weakly connected", as opposed to "strongly connected", where one requires the existence of a path between any two given vertices. We shall see in Remark~\ref{rk:connected} that these two notions coincide for the graphs of interest in this paper.

The following result allows us to determine whether a graph covering arising from a voltage assignment is Galois.

\begin{proposition}\label{prop:voltage}
Let $X$ be a directed graph, $G$ a finite group equipped with a $G$-valued voltage assignment $\alpha$ on $X$. Let $Y$ denote the derived graph $X(G,\alpha)$.
\item[i)]If $X$ is connected, then the natural action of $G$ on $Y$ permutes transitively the connected components of $Y$.
\item[ii)]If  $Y$ is connected, then $Y/X$ is a Galois covering. Furthermore, its Galois group is isomorphic to $G$.
\end{proposition}
\begin{proof}
    This follows from \cite[Theorem~2.10]{gonet22} and \cite[\S2.3]{gonet-thesis}.
\end{proof}

\begin{corollary}\label{cor:not-Galois}
    Let $X$, $Y$ and $G$ be as in Proposition~\ref{prop:voltage}. Let $d$ be the number of connected components of $Y$. If $d!>|G|$, then $Y/X$ is not Galois.
\end{corollary}
\begin{proof}
    Let $Y_1$ and $Y_2$ be two connected components of $Y$. Proposition~\ref{prop:voltage}i) tells us that there exists $g\in G$ such that 
    \[
         V(Y_1)\ni(v,\sigma)\mapsto (v,g\cdot \sigma),\quad  \EE(Y_1)\ni(e,\sigma)\mapsto (e,g\cdot \sigma)
    \]
    induces an isomorphism of graphs from $Y_1$ to $Y_2$. Thus, any permutation of the connected components of $Y$ gives rise to an element of $\Deck(Y/X)$. Since there are $d!$ such permutations and the degree of the covering is $|G|$, the covering is not Galois if $d!>|G|$.
\end{proof}

\begin{remark}
    \label{rem:connected-comp} 
    Let $G$ be a finite group. In the course of this article, we frequently study the number of connected components of $X(G,\alpha)$. To do so, it suffices to fix a vertex $v\in V(X)$ and study the set \[\left\{ (v,g)\in V(X(G,\alpha))\mid (v,g) \textup{ lies in the same connected component as $(v,1)$}\right\}.\]
    Let $d_v$ be the cardinality of this set.  As explained in the proof of Corollary~\ref{cor:not-Galois}, all connected components of $X(G,\alpha)$ are isomorphic to each other. Therefore, the number of connected components of $X(G,\alpha)$ is given by $\frac{\vert G\vert }{d_v}$.
\end{remark}

\subsection{Realizing isogeny graphs as voltage graphs}
\label{sec:realize}

The goal of this section is to realize the graph $X_l^q(p^nN)$ as a voltage graph arising from a $\GL_2(\ZZ/p^n\ZZ)$-valued voltage assignment on $X$, similar to the discussion in \cite[Appendix A]{LM1}.

For each elliptic curve $E\in S$, we fix a $\Zp$-basis $\{s_E,t_E\}$ of the Tate module $T_p(E)$.
Let $e\in \EE(X_l^q(N))$ with source $(E,R_1,R_2)$ and target $(E',R_1',R_2')$. It arises from some $l$-isogeny $\phi:E\rightarrow E'$. Since $p$ and $l$ are coprime, $\phi$ induces a  $\Zp$-isomorphism
\[
\phi:T_p(E)\rightarrow T_p(E').
\]
This allows us to define the following voltage assignment:
\begin{defn}\label{def:alpha_n}
For any given $e\in \EE(X_l^q(N))$ arising from an $l$-isogeny $\phi:E\rightarrow E'$, we define $g_e\in \GL_2(\Zp)$ to be the transpose of the matrix of $\phi$ with respect to our chosen bases $\{s_E,t_E\}$ and $\{s_{E'},t_{E'}\}$ so that the following equation holds
\[
\begin{pmatrix}
    \phi(s_E)\\\phi(t_E)
\end{pmatrix}=g_e\cdot
\begin{pmatrix}
    s_{E'}\\t_{E'}
\end{pmatrix}.
\]

    For an integer $n\ge 0$, we define $\alpha_n:\EE(X_l^q(N))\rightarrow \GL_2(\ZZ/p^n\ZZ)$ to be the voltage assignment sending $e$ to  the image of $g_e$ under the natural projection $\GL_2(\Zp)\rightarrow\GL_2(\ZZ/p^n\ZZ)$.
\end{defn}
Let $(E,R_1,R_2,P,Q)\in V(X_\ell^q(p^nN))$. Since $\{P,Q\}$ is a basis of $E[p^n]$, there exists a unique $\sigma\in \GL_2(\ZZ/p^n\ZZ)$ such that
    \begin{equation}
    \begin{pmatrix}
        P\\Q
    \end{pmatrix}=\sigma\cdot\begin{pmatrix}
        \overline{s}_E\\\overline{t}_E
    \end{pmatrix},\label{eq:condition-basis}    
    \end{equation}
        where $\overline{s}_E$ and $\overline{t}_E$ denote the images of $s_E$ and $t_E$ in $E[p^n]$, respectively. This induces a well-defined injective map
        \begin{align*}
            \Phi_n\colon V(X_\ell^q(p^nN))&\to V\left(X(\GL_2(\ZZ/p^n\ZZ),\alpha_n)\right)\\
            (E,R_1,R_2,P,Q)&\mapsto \big((E,R_1,R_2),\sigma\big).
        \end{align*}
\begin{theorem}\label{thm:iso}
For all $n\ge0$, the map $\Phi_n$ defines an isomorphism of graphs \[\Phi:X_l^q(p^nN)\stackrel\sim\to X(\GL_2(\ZZ/p^n\ZZ),\alpha_n).\]
\end{theorem}
\begin{proof}
    To simplify notation, we write  $Y_n=X(\GL_2(\ZZ/p^n\ZZ),\alpha_n)$, $Z_n=X_l^q(p^nN)$ and $G_n=\GL_2(\ZZ/p^n\ZZ)$ in this proof.
    
   By definition, $\Phi_n$ is injective on $V(Z_n)$. Given any $\left((E,R_1,R_2),\sigma\right)\in V(Y_n)$, we may define a basis of $E[p^n]$ via \eqref{eq:condition-basis}. Hence,  $\Phi$ is a bijection from $V(Z_n)$ to $V(Y_n)$.

    It remains to show that the bijection $\Phi_n$ respects the edges of the two graphs. That is, an isogeny $\phi$ induces an edge in $Z_n$ from $v$ to $w$  if and only if it induces an edge in $Y_n$ from $\Phi_n(v)$ to $\Phi_n(w)$.  Suppose that $e\in \EE(Z_0)$ with $(E,R_1,R_2)$ and $(E',R'_1,R'_2)$ as the source and target, respectively. Let $\phi$ be the corresponding $l$-isogeny inducing $e$. The same isogeny gives rise to an edge from $(E,R_1,R_2,P,Q)$ to $(E',R'_1,R'_2,P',Q')$ in $Z_n$ if and only if 
    \begin{equation}
        \phi(P)=P'\quad \text{and} \quad\phi(Q)=Q'.\label{eq:edge-condition}
    \end{equation}
    Let us write 
    \begin{align*}
        \Phi_n((E,R_1,R_2,P,Q))&=((E,R_1,R_2),\sigma),\\
          \Phi_n((E',R'_1,R'_2,P',Q'))&=((E',R'_1,R'_2),\sigma').
    \end{align*}
    A direct calculation shows that \eqref{eq:edge-condition} is equivalent to the equation
    \[\sigma\alpha_n(e)=\sigma',\]
    which is precisely the condition for there to be an edge going from $((E,R_1,R_2),\sigma)$ to $((E',R'_1,R'_2),\sigma')$ in $Y_n$. This shows that $\Phi_n$ respects the edges of $Y_n$ and $Z_n$, as desired.
\end{proof}

\begin{corollary}\label{cor:cover}
    Let $n\ge1$ be an integer. The map 
    \begin{align*}
    V(X_l^q(p^nN))&\rightarrow V(X_l^q(p^{n-1}N)),\\ (E,R_1,R_2,P,Q)&\mapsto (E,R_1,R_2,pP,pQ)
    \end{align*}
     induces a graph covering $X_n\rightarrow X_{n-1}$.    
\end{corollary}
\begin{proof}
Let $G_n=\GL_2(\ZZ/p^n\ZZ)$ and let $\alpha_n$ be the voltage assignment on $X=X_l^q(N)$ given in Definition~\ref{def:alpha_n}. Let $\pi_n:G_n\rightarrow G_{n-1}$ be the natural projection map and $N_n=\ker(\pi_n)$. {Let $\tilde{\alpha_n}=\pi_n\circ\alpha_n$}. Then Lemma~\ref{lem:intermediate} tells us that there is an intermediate graph covering $X(G_n,\alpha_n)/ X(G_n/N_n,\tilde\alpha_n)$ of $X(G_n,\alpha_n)/X$.

 Let $\Theta_n:V(X(G_n,\alpha_n))\rightarrow V(X_l^q(p^nN))$ be the inverse of $\Phi_n$.
 Since $G_{n-1}\cong G_n/N_n$ and $\tilde\alpha_n=\alpha_{n-1}$, we may identify $X(G_n/N_n,\tilde\alpha_{n})$ with $X(G_{n-1},\alpha_{n-1})$, and therefore with $X_l^q(p^{n-1}N)$. 
 Finally, it follows from the definition of the Tate module that
 \[
 \Theta_{n-1}\circ \pi_n=[p]\circ \Theta_{n},
 \]
 which concludes the proof.
\end{proof}

\begin{corollary}\label{cor:Galois}
Let $X$ be a connected component of $X_l^q(N)$ (in the sense of Definition~\ref{def:connected}). For an integer $n\ge1$, let $X_n$ denote the pre-image of $X$ in $X_l^q(p^nN)$ under the natural projection map. If $X_n$ is connected, the covering $X_n/X$ is Galois whose Galois group is isomorphic to $\textup{GL}_2(\ZZ/p^n\ZZ)$.
\end{corollary}
\begin{proof}
    This follows directly from Proposition~\ref{prop:voltage}ii).
\end{proof}

We conclude this section with a couple of remarks on the connected components of $X_l^q(p^nN)$, which will be useful for subsequent sections.

\begin{remark}\label{rk:multiple-edges}
    The graph $X_l^q(p^nN)$ contains multiple edges if and only if there are two distinct $l$-isogenies $\phi,\psi\colon E\to E'$, which coincide on $E[p^nN]$. Given any two $l$-isogenies $\phi$ and $\psi$, we have $
    \deg(\phi-\psi)\le 4l$. Thus,  $\ker(\phi-\psi)$ does not contain any elements of order $p^nN$ whenever $p^nN>4l$. In other words, $X_l^q(p^nN)$ does not contain multiple edges if $n$ is sufficiently large. 
    \end{remark}

\begin{remark}\label{rk:connected}
    Let $X_n$ be a connected component of $X_l^q(p^nN)$. Let $e\in \EE(X_n)$ with $v=(E,R_1,R_2,P,Q)$ and $w=(E',R'_1,R'_2,P',Q')$ as the source and target, respectively. Then $e$ arises from an $l$-isogeny $\phi:E\rightarrow E'$ such that $\phi(U)=U'$ for $U\in\{R_1,R_2,P,Q\}$. The dual isogeny $\hat\phi$ gives rise to an edge in $X_n$ from $w$ to $(E,lR_1,lR_2,lP,lQ)$. Let $d$ be the order of $[l]$ as an automorphism on $E[p^nN]$. In particular, $l^d U=U$ for all $U\in\{R_1,R_2,P,Q\}$. On repeatedly applying $\phi$ and $\hat\phi$ alternatively, we obtain a path of the form 
    \begin{align*}
    w\stackrel{\hat\phi}\to &(E,lR_1,lR_2,lP,lQ)\stackrel{\phi}\to(E',lR_1',lR_2',lP',lQ')\stackrel{\hat\phi}\to\\
    &(E,l^2R_1,l^2R_2,l^2P,l^2Q)\stackrel\phi\to\cdots\stackrel{\hat\phi}\to(E,l^{d}R_1,l^{d}R_2,l^{d}P,l^{d}Q)=v.
    \end{align*}
In other words,  each directed edge in $X_n$ can be "reversed" via a path in the opposite direction. Hence,  $X_n$ is a strongly connected directed graph.
\end{remark}

\begin{defn}
   An isogeny of elliptic curves $\phi:E\to E'$ is said to be an \textbf{$l$-power isogeny} if the degree of $\phi$ is a power of $l$. If $E=E'$, we say that $\phi$ is an \textbf{$l$-power endomorphism} on $E$.
\end{defn}

\begin{remark}
    \label{rem:path}
    Let $v=(E,R_1,R_2,P,Q)$ and $w=(E',R'_1,R'_2,P',Q')$ be two vertices of $X_l^q(p^nN)$. Since every $l$-power isogeny can be realized as the composition of $l$-isogenies, it follows from the definition that there is a path from $v$ to $w$ if and only if there exists an $l$-power isogeny $\phi\colon E\to E'$  such that $\phi(U)=U'$ for $U\in\{R_1,R_2,P,Q\}$.  
\end{remark}

\section{The ordinary case}

The goal of this section is to prove Theorem~\ref{thmA} stated in the introduction.
We fix an ordinary connected component $X$ of $X_l^q(N)$ and define $X_n$ to be the pre-image of $X$ in $X_l^q(p^nN)$ under the natural projection map. As explained in the introduction, all elliptic curves that give rise to a vertex in $X_n$ {are ordinary}. Furthermore, since these curves are isogenous to each other, their endomorphism rings are orders in the same imaginary quadratic field, which we denote by $K$. We refer to $K$ as the CM field of $X$. Note that the prime $r$ is split in $K$.

We begin with the following preliminary lemmas.

\begin{lemma}\label{lem:order}
Let $n\ge1$ be an integer.    The order of the group $\textup{GL}_2(\ZZ/p^n\ZZ)$ is $p^{4(n-1)}(p^2-1)(p^2-p)$.
\end{lemma}
\begin{proof}Let $G_n=\GL_2(\ZZ/p^n\ZZ)$. Let $\pi_n:G_n\rightarrow G_{1}$ be the group homomorphism induced by the projection $\ZZ/p^n\ZZ\rightarrow \ZZ/p\ZZ$. Then $\pi_n$ is surjective with $$\ker\pi_n=I_2+p M_{2\times 2}(\ZZ/p^n\ZZ).$$ Thus, $|\ker(\pi_n)|=p^{4(n-1)}$, and $|G_n|=p^{4(n-1)}|G_{1}|$. It is a standard fact that $|G_1|=(p^2-1)(p^2-p)$. Hence the lemma follows.
   \end{proof}
   
  \begin{lemma}[{\cite[Theorem 4]{interlando-elia-i}, \cite[Theorems 2 and 3]{interlando-elia-ii}}]\label{structure-residue}
  Let $A$ be the ring of integers of a finite extension of $\Q_p$, and let $\varpi$ be a uniformizer of $A$. Let $f$ and $e$ be the degree of inertia and the degree of ramification of the extension, respectively.  Let $G_a$ be the multiplicative group $(A/\varpi^a)^\times$, where $a\ge1$ is an integer.
  \begin{itemize}
  \item If $e=1$ and $p\ne2$,  
  \[G_a\cong \ZZ/(p^f-1)\ZZ\times (\ZZ/p^{a-1}\ZZ)^{f},\]
  \item If $e=1$, $p=2$ and $a\ge 3$,  
  \[
  G_a\cong\ZZ/(2^f-1)\ZZ\times (\ZZ/2\ZZ)\times \ZZ/2^{a-2}\ZZ\times(\ZZ/2^{a-1}\ZZ)^{f-1}.
  \] 
  \item If $e=2$, $p\ne 2$ and $a\ge3$ is odd,  
  \[
  G_a\cong\ZZ/(p^f-1)\ZZ\times  (\ZZ/p^{(a-1)/2}\ZZ)^{2f} .
  \]
  \item If $e=2$, $p\ne 2$ and $a\ge4$ is even,   
  \[
  G_a\cong\ZZ/(p^f-1)\ZZ\times  (\ZZ/p^{a/2-1}\ZZ)^{f} \times  (\ZZ/p^{a/2}\ZZ)^{f}.
  \]   
  \item If $e=2$, $p=2$ and $a>4$, 
  \[
  G_a\cong\ZZ/(2^f-1)\ZZ\times (\ZZ/2^{h_1}\ZZ)^{f}\times (\ZZ/2^{h_2})^f,
  \]
  where $h_1=\lceil (a-1)/2\rceil$ and $h_2=\lceil(a-2)/2\rceil$.
  \end{itemize}
  \end{lemma}

\begin{corollary}\label{cor:order-mod}
     Let $A$ be the ring of integers of a finite extension of $\Qp$ with ramification index $e\in\{1,2\}$ and degree of inertia $f$. Let $n\ge1$ be an integer, and let $\alpha\in A^\times$. There exists a constant $d\le p^f-1$ such that for $n$ sufficiently large, the order of the image of $\alpha$ in $(A/p^n)^\times$ is equal to $dp^{n-1}$ (resp. $dp^{n}$) if $e=1$ (resp. $e=2$).   
\end{corollary}
\begin{proof}
We consider the case $p\ne 2$. The topological group $A^\times$ is isomorphic to $C:=\ZZ/(p^f-1)\times\Zp^{\oplus ef}$. Let $g_1,\ldots, g_{ef}\in A^\times$ be the elements that correspond to the elements $(0,1,0,\dots,0)$, $\dots $, $(0,0,\dots,0,1)$ in $C$.

We have $\varpi^e A=pA$. By Lemma~\ref{structure-residue}, the order of $g_i$ in $(A/p^n)^\times$ is of the form $p^{t_{n,i}}$, where $t_{n,i}=n-1$ (resp. $t_{n,i}\in\{n,n-1\}$) if $e=1$ (resp. $e=2$). Let $(\alpha_0,\alpha_1,\dots ,\alpha_{ef})$ be the image of $\alpha$ in $C$. Let $k_i=\ord_{{g_i}}(\alpha_i)$ for $i=1,\ldots ,ef$, and let $d$ be the order of $\alpha_0$ (as an element in $\ZZ/(p^f-1)\ZZ$). Then, the order of the image of $\alpha$ in $(A/\varpi^n)^\times$ is given by 
\[
\mathrm{lcm}(d,p^{t_{n,1}-k_1},\dots,p^{t_{n,ef}-k_{ef}}),
\]
from which the corollary follows. The case $p=2$ can be shown in a similar way.\end{proof}

\begin{theorem} \label{thm:ordinary-components}
    Let $X$ be an ordinary connected component of $X_l^q(N)$, and let $K$ be the CM field of $X$. For each integer $n\ge1$, let $X_n$ be the pre-image of $X$ in $X_l^q(p^nN)$. For $n$ sufficiently large, the number of connected components in $X_n$ is given by $cp^{2(n-1)}$ (resp. $cp^{3(n-1)}$) if $p$ is split (resp. non-split) in $K$, where $c\ge (p^2-p)$ is a constant independent of $n$.
\end{theorem}
\begin{proof}
    The degree of the covering $X_l^q(p^nN)/X_l^q(N)$ equals $|\GL_2(\ZZ/p^n\ZZ)|$. As explained in Remark~\ref{rem:connected-comp}, the number of connected components of $X_n$ is given by $|\GL_2(\ZZ/p^n\ZZ)|$ divided by the number of pre-images of a fixed $(E,R_1,R_2)\in V(X)$ lying in the same connected component of $X_n$.  
    
    As all curves with complex multiplication by an order in $K$ are isogenous, we may assume that $\textup{End}(E)\cong\mathcal{O}$ is an order such that $[\mathcal{O}_K:\mathcal{O}]$ is not divisible by $l$. In particular, the splitting behaviour of $l$ in $\mathcal{O}$ is the same as that in $\mathcal{O}_K$.

    Let $(E,R_1,R_2,P,Q)$ and $(E,R_1,R_2,P',Q')$ be two pre-images of $(E,R_1,R_2)$ in $V(X_n)$. By Remark~\ref{rem:path}, they lie in the same connected component of $X_n$ if and only if there is an $l$-power isogeny $\phi\colon E\to E$ such that $\phi(R_1)=R_1$, $\phi(R_2)=R_2$, $\phi(P)=P'$ and $\phi(Q)=Q'$.
    
    As discussed in \cite[Remark~2.4]{LM1}, there exists an elliptic curve $\mathbf{E}$ defined over some number field $L$ with complex multiplication by $\mathcal{O}$ such that $\mathbf{E}\pmod{\mathfrak{r}}=E$, where $\mathfrak{r}$ is a prime ideal above $r$ in $L$.  There is a group isomorphism $E[p^n]\cong \mathbf{E}[p^n]$. As $\mathbf{E}(\CC)\cong\CC/\cO$, we have, furthermore $E[p^n]\cong\cO/(p^n)$. Therefore,  under the isomorphism $\End(E)\cong\cO$, the {action of the set of aforementioned $l$-power isogenies on $E[p^n]$} can be identified with a subgroup of $(\cO/p^n)^\times$.
    
    Let $U$ be the (multiplicative) set of elements in $\mathcal{O}$ of $l$-power norm that act trivially on $E[N]$ (since we are only interested in the endomorphisms that fix $R_1$ and $R_2$). 
   Let $U_n$ be the image of $U$ in $(\mathcal{O}/p^n)^\times$. The number of connected components in $X_n$ is then given by $|\GL_2(\ZZ/p^n\ZZ)|/|U_n|$. Lemma~\ref{structure-residue} tells us that $|(\mathcal{O}/p^n)^\times|$ is given by $(p-1)^2p^{2(n-1)}$, $(p^2-1)p^{2(n-1)}$ or $(p-1)p^{2n-1}$, depending on whether $p$ is split, inert, or ramified in $\cO$. In all three cases, we have $|(\mathcal{O}/p^n)^\times|\le (p^2-1)p^{2(n-1)}$.

   If $l$ is split in $K$, then $U_n$ is generated by two elements, and $|U_n|\le|(\mathcal{O}/p^n)^\times|\le (p^2-1)p^{2(n-1)}$. If $l$ is non-split in $K$, then $U_n$ is cyclic. Thus, it follows from Corollary~\ref{cor:order-mod} that $|U_n|$ is bounded above by $(p-1)p^{n-1}$, $(p^2-1)p^{n-1}$ or $(p-1)p^n$, depending on whether $p$ is split, inert or ramified in $\cO$. In all cases, we have $|U_n|\le (p^2-1)p^{n-1}$.
  Furthermore, Corollary~\ref{cor:order-mod} tells us that the order of the generators grows linearly in $p^n$. 
   Hence, the theorem follows from Lemma~\ref{lem:order}.  
\end{proof} 

We are now ready to prove Theorem~\ref{thmA}(i).
\begin{corollary} \label{cor:ord-not-Galois}
    The covering $X_{n}/X$ is not Galois for $n$ sufficiently large.
\end{corollary}
\begin{proof}
    We prove that $X_{n}/X_{n-1}$ is not Galois, which in turn implies that $X_{n}/X$ is not Galois.
    
    It follows from Theorem~\ref{thm:ordinary-components} that the number of connected components of $X_{n}$ lying above $X_{n-1}$ is at least $p^2$ when $n$ is sufficiently large. Since the degree of the covering $X_{n}/X_{n-1}$ is $p^4$ by Lemma~\ref{lem:order}, it suffices to show that $(p^2)!>p^4$ in light of Corollary~\ref{cor:not-Galois}.

    Indeed,
    \[(p^2-2)(p^2-3)>\frac{4}{3}\ge\frac{1}{1-\frac{1}{p^2}}.\]
    This in turn gives
    \[
    (p^2)!\ge p^2(p^2-1)(p^2-2)(p^2-3)>p^4,
    \]
    as required.
\end{proof}

We conclude this section with the proof of Theorem~\ref{thmA}(ii):
\begin{theorem}\label{thm:abelian}
   Let $X$ be an ordinary connected component of $X_l^q(N)$, and let $K$ be the CM field of $X$. Let $X_0'=X$. For each integer $n\ge1$, let $X_n'$ be a connected component of the pre-image of $X_{n-1}'$ in $X_l^q(p^nN)$. Then $X_n'/X$ is a Galois covering. Furthermore, $\displaystyle\varprojlim_n\Gal(X_n'/X)$ is an abelian $p$-adic Lie group of dimension 2 (resp. dimension 1) if $l$ is split (resp. non-split) in $K$.
\end{theorem}
\begin{proof}
It follows from Theorem~\ref{thm:iso} and Proposition~\ref{prop:voltage} that $\GL_2(\ZZ/p^n\ZZ)$ acts transitively on the set of connected components of $X_n$. 
Let $H_n\subset \textup{GL}_2(\ZZ/p^n\ZZ)$ be the stabilizer of $X'_n$ under this action. From the orbit-stabilizer formula it follows that $X_n'/X$ is a $|H_n|$-sheeted covering. Furthermore, the action of $H_n$ on $X_n'$ gives rise to an injection $\Phi:H_n\hookrightarrow \Deck(X_n'/X)$.

Fix a vertex $v=(E,R_1,R_2,P,Q)\in V(X_n')$. 
Let $\sigma\in \textup{Deck}(X'_n/X)$. Then $\sigma(v)$ is of the form $(E,R_1,R_2,P',Q')$. As $X'_n$ is connected, and {$e$ and $\sigma(e)$ are induced by the same isogeny}, $\sigma$ is uniquely defined by $\sigma(v)$. Therefore, $\Phi$ is surjective. In particular, $X_n'/X$ is Galois.

The action of $H_n$ on $X_n'$ also sends $v$ to a vertex of the form $(E,R_1,R_2,P',Q')$. In particular, the action of an element in $H_n$ can be realized as an $l$-power isogeny from $E$ to $E$. This implies that $H_n$ is isomorphic to the group $U_n$ defined in the proof of Theorem~\ref{thm:ordinary-components}.
Note that $\displaystyle\varprojlim_n U_n$ is a subgroup of $\displaystyle\varprojlim_n(\cO/p^n)^\times$, which is a $p$-adic Lie group of dimension $2$. If $l$ is non-split in $K$, the subgroup $\displaystyle\varprojlim_n U_n$ is generated by a non-torsion elements. This implies the assertion of the theorem.

Suppose that $\ell$ is split in $K$. Let $K(p^n)$ be the ray class field over $K$ of conductor $p^n$. The Frobenii of the two primes above $\ell$ in $\Gal\left(\bigcup_n K(p^n)/K\right)$ generate a $p$-adic Lie group of dimension $2$ by class field theory. Furthermore, they induce two elements $\sigma_1,\sigma_2\in\displaystyle\varprojlim_n (\cO/p^n)^\times$, both of which have norm $l$. As $E[N]$ is a finite group, there exists an integer $m$ such that the action of $\sigma_i^m$ on $E[N]$ is trivial for $i=1$ and $2$. In particular, $\sigma_1^m,\sigma_2^m\in\displaystyle\varprojlim_n U_n$, and they generate a $p$-adic Lie group of dimension 2. In particular, this implies that $\displaystyle\varprojlim_n U_n$ itself is a $2$-dimensional $p$-adic Lie group, as desired.
\end{proof}

\section{The supersingular case}\label{sec:ss}
We now turn our attention to the supersingular case and prove Theorem~\ref{thmB}. Let $X^{\ss}$ be the supersingular subgraph of $X_l^q(N)$ and define $X^{\ss}_n$ as its pre-image in $X_l^q(p^nN)$. 
Let $E$ be any supersingular elliptic curve that gives rise to a vertex in $X$. Let $B=\textup{End}(E)\otimes \Q$. Then $B$ is a quaternion algebra, which is only ramified at $r=\textup{char}(\mathbb{F}_q)$ and $\infty$. We write $\mathrm{Nm}$ for the reduced norm map on $B$. The endomorphism ring $\textup{End}(E)$ is a maximal order in $B$. As each isomorphism class of supersingular elliptic curves has a representative defined over $\mathbb{F}_{{r}^2}$, we assume $\FF_{r^2}\subset\FF_q$ in this section. {In what follows, $\mathbb{A}_{\Q}^f$ denotes the finite part of the ring of adèles over $\QQ$.}
\begin{lemma}[{Strong approximation}]\label{lem:approximation}
    Let $G$ be the algebraic group such that $G(\Q)=B^1=\{x\in B\mid \textup{Nm}(x)=1\}$. Let $y=(y_u)_u\in G(\mathbb{A}_{\Q}^f)$ be an element such that $y_u=1$ for all $u\nmid Np$ and $y_u\in \textup{SL}_2(\Z_u)$ for $u|Np$. There exists an element $x\in G(\Q)$ that is arbitrarily close to $y_u$ at all $u\ne l$. 
\end{lemma}
\begin{proof}
    See  \cite[proof of Theorem~5.3.3]{gorenkassaei}.
\end{proof}

\begin{corollary}\label{cor:approximation}Let $C\in \textup{SL}_2(\ZZ/p^nN\ZZ)$. There exists $x\in B^1$ such that
\begin{itemize}
    \item $x$ is integral at all places  $u\ne l$,
    \item $x$ acts as $C$ on $E[p^nN]$.
\end{itemize}\end{corollary}
\begin{proof}
For each prime $u\mid p^nN$, let $C_u$ be an element of $\textup{SL}_2(\ZZ_u)$ such that $C_u\equiv C\pmod{u^{\ord_u(p^nN)}}$. 
    Let $(y_u)_u\in G(\mathbb{A}_{\Q}^f)$ be an element such that $y_u=1$ for all $u\nmid p^nN$ and that $y_u=C_u$ for all $u\mid p^nN$. By Lemma \ref{lem:approximation}, there exists $x\in G(\Q)$ that is arbitrarily close to $y$, except possibly at $l$. As $y$ is integral at all places, this implies that $x$ is also integral at all places, except possibly at $l$. As $l$ is coprime to $p^nN$, $x$ approximates the matrix $C$ at all places $u\mid p^nN$. Thus, $x$ acts as $C$ on $E[p^nN]$. 
\end{proof}
\begin{theorem}[Roda]
\label{thm:number-connected-comp}
    The number of connected components of $X^{\ss}_n$  is given by the index $[(\ZZ/p^nN\ZZ)^\times:\langle l\rangle]$. In particular, the number of connected components is uniformly bounded as $n$ varies.
\end{theorem}
This theorem is the main result of \cite[Section 3.3]{thesis-roda}.  We outline the proof here for the convenience of the reader.
\begin{proof}
    Let $(E,R_1,R_2,P,Q)$ and $(E,R'_1,R'_2,P',Q')$ be two vertices in $X^{\ss}_n$. There is a matrix $C$ in $\textup{GL}_2(\ZZ/p^nN\ZZ)$ such that \[C\begin{pmatrix}
        R_1 +P\\R_2+Q
    \end{pmatrix}=\begin{pmatrix}R'_1+P'\\R'_2+Q'\end{pmatrix}.\] 
    In light of Remark~\ref{rem:path}, we determine when $C$ is induced by an $l$-power endomorphism. 

We claim that $C\in \textup{SL}_2(\ZZ/p^nN\ZZ)$ induces an $l$-power endomorphism on $E$ if and only if $\det(C)\in\langle l\rangle$ as an element in $(\ZZ/p^nN\ZZ)^\times$. The only if part is clear. We show the if part below.
    
    First, assume that $C\in \textup{SL}_2(\ZZ/p^nN\ZZ)$. By Corollary \ref{cor:approximation}, there exists $x\in B$ of norm one that is integral at all places except possibly at $l$ and such that $x$ acts as $C$ on $E[p^nN]$. 
    Let $t$ be a positive integer such that $l^t$ acts trivially on $E[p^nN]$ and such that $l^tx$ is integral at all places. Then $y=l^tx$ is an $l$-power endomorphism of $E$ such that $y(P+R_1)=P'+R'_1$ and $y(Q+R_2)=Q'+R'_2$.
    
    Assume now that $\det(C)$ is a power of $l$. By \cite[proof of Theorem~5.33]{gorenkassaei}, $E$ admits an endomorphism $f$ of degree $l^t$ with $t$ odd. As $[l]$ has degree $l^2$, it follows that there is a matrix $D\in \textup{GL}_2(\ZZ/p^nN\ZZ)$ induced by an $l$-power endomorphism such that $\det(C)=\det(D)$. Thus, to determine whether $C$ is induced by an $l$-power endomorphism, it suffices to consider $C'=CD^{-1}$. As $\det(C')=1$, we are reduced to the previous case. This concludes the proof of our claim. Combined with Remark~\ref{rem:connected-comp}, this implies that the number of connected components is given by $[(\ZZ/p^nN\ZZ)^\times:\langle l \rangle]$. 

    Suppose that $p\ne2$. There is a group isomorphism $$(\ZZ/p^nN\ZZ)^\times\cong (\ZZ/Np\ZZ)^\times \times (1+p\Zp)/(1+p^{n-1}\Zp).$$ Let $\omega_p$ be the Teichm\"uller character and write $\tilde l=l/\omega_p(l)\in 1+p\Zp$. Let $\overline{\langle \tilde l\rangle}$ denote the topological closure of $\langle \tilde l \rangle$ within $1+p\Z_p$. Note that the multiplicative group $1+p\Z_p$ is isomorphic to the additive group $\Zp$. As $\omega_p(l)$ is a $(p-1)$-st root of unity, it follows that $\tilde l\neq1$. In particular, the index $[1+p\Zp:\overline{\langle \tilde l\rangle}]$ is finite. Therefore, we deduce the inequality
    $$[(\ZZ/p^nN\ZZ)^\times:\langle l \rangle]\le [(\ZZ/Np\ZZ)^\times:\langle l \rangle][1+p\Zp:\overline{\langle \tilde l\rangle}],$$ where the upper bound is independent of $n$.  Thus, the number of connected components in $X_n^\ss$ is bounded independently of $n$. The case $p=2$ can be proved in a similar way.
\end{proof}

\begin{corollary}\label{cor:density}
    Assume that $p>2$.  Let $T$ be the set of primes $l$ such that $X^{\ss}_n$ is connected. Let $\cT$ be the set of elements $a\in(\ZZ/p^2N\ZZ)^\times$ such that $\langle a\rangle=(\ZZ/p^2N\ZZ)^\times$. Then the Dirichlet density of $T$ is equal to $|\cT|/\varphi(p^2N)$, where $\varphi$ is the Euler totient function. \end{corollary}
    \begin{proof}
        Theorem~\ref{thm:number-connected-comp} says that the graph $X_n$ is connected if and only if $l$ generates $(\ZZ/p^nN\ZZ)^\times$ for all $n$. This is the case if and only if $l$ generates the cyclic group $(\ZZ/p^2N\ZZ)^\times$, from which the corollary follows.
    \end{proof}
    \begin{remark}
    \label{cyclic-for-large-N}
    Note that $\cT$ is non-empty if and only if $(\ZZ/p^2N\ZZ)^\times$ is cyclic. This is the case when $p$ is odd, $\varphi(N)$ is coprime to $p(p-1)$ and $(\ZZ/N\ZZ)^\times$ is cyclic. 
     If $p$ is odd, for $\varphi(N)$ and $p(p-1)$ to be coprime, we must have $N\in\{1,2\}$. In this case, the density given by Corollary~\ref{cor:density} is $\varphi(p-1)/p$.
        If $p=2$, the group $(\ZZ/p^nN\ZZ)^\times$ is not cyclic. In particular, $X_n$ is not connected for any choice of $l$.
    \end{remark}

    \begin{corollary}
    \label{cor:connected}
        Assume that $X^{\ss}_n$ is connected for all $n$. Then $X_n^\ss/X^\ss$ is a Galois covering whose Galois group isomorphic to $\textup{GL}_2(\ZZ/p^n\ZZ)$.
    \end{corollary}
 
    \begin{proof}
        This follows immediately from Proposition~\ref{prop:voltage}ii).
    \end{proof}

Finally, we show a supersingular analogue of Theorem~\ref{thm:abelian}. That is, we show that even if the graphs $X_n^\ss$ are not connected, we may still obtain a tower of Galois coverings.

    \begin{corollary}
    \label{cor:galois-supersingular}
        There exists an integer $m_0$ such that the number of connected components of $X_n^\ss$ stabilizes. Let $Z_{m_0}$ be a connected component of $X^{\ss}_{m_0}$. For $n\ge m\ge m_0$, let $Z_n$ be the inverse image of $Z_{m_0}$ in $X_n$, and let $G_{n,m}\subset \textup{GL}_2(\ZZ/p^n\ZZ)$ be the subgroup of matrices congruent to $\begin{pmatrix}
            1&0\\0&1
        \end{pmatrix} \pmod {p^m}$. Then $Z_n/Z_m$ is Galois with $\Gal(Z_n/Z_m)\cong G_{n,m}$. In particular, the inverse limit $\displaystyle\varprojlim_{n\ge m_0} G_{n,m_0}$ is an open subgroup of $\GL_2(\Zp)$. 
    \end{corollary}
    \begin{proof}
The existence of $m_0$ is given by Theorem~\ref{thm:number-connected-comp}. 
Let $n\ge m\ge m_0$. Since the number of connected components stabilizes, $Z_n$ and $Z_m$ are connected. Furthermore, $Z_n/Z_m$ is a $\vert G_{n,m}\vert $-sheeted covering and there is a natural embedding
        \[\psi_m\colon G_{n,m}\hookrightarrow \textup{Deck}(Z_n/Z_m).\]
        It remains to show that this embedding is indeed an isomorphism.
        
        Let $\sigma\in\textup{Deck}(Z_n/Z_m)$ and let $v$ be a vertex of $Z_n$. Let $\pi \colon Z_n\to Z_m$ be the projection map. By definition
        \[\pi(\sigma(v))=\pi(v).\]
        The graph $Z_n$ is a subgraph of $X^{\ss}_n$. Therefore, there exists a vertex $v_0$ of $X^{\ss}$ and elements $h,h'\in\textup{GL}_2(\ZZ/p^n\ZZ)$ such that $v=(v_0,h)$ and $\sigma(v)=(v_0,h')$. As $\pi(\sigma(v))=\pi(v)$, the elements $h$ and $h'$ have the same image in $\textup{GL}_2(\ZZ/p^m\ZZ)$. In particular, there exists an element $g\in G_{n,m}$ such that $h'=gh$, which is equivalent to $g(v)=\sigma(v)$. We show in the following that the action of $g$ on $Z_n$ coincides with that of $\sigma$.
        
        Let $w$ be a vertex of $Z_n$ that admits an edge whose source and target are $v$ and $w$, respectively. We need to show that $\sigma(w)=g(w)$.  
       For each edge from $\pi(v)$ to $\pi(w)$, we fix an isogeny $\phi$ that induces this edge. Each such isogeny induces exactly one edge in $Z_n$ from $\sigma(v)$ to some $w'\in V(Z_n)$ such that $\pi(w)=\pi(w')$. 
        As $\sigma$ and $g$ are deck transformations of $Z_n/Z_m$, we have $\sigma(w)=g(w)=w'$. Since $Z_n$ is strongly connected by Remark~\ref{rk:connected}, this property propagates to all vertices of $Z_n$. Thus, the actions of $\sigma$ and $g$ agree as desired, proving that $\psi_m$ is indeed an isomorphism.
 \end{proof}

\begin{remark}
\label{rem:subgroup}
The inverse limit $\displaystyle\varprojlim_{n\ge m_0} G_{n,m_0}$ is equal to $\GL_2(\Zp)$ if and only if $m_0=0$, which is equivalent to $p$ being odd and $[(\ZZ/Np^2\ZZ)^\times:\langle \ell\rangle]=[(\ZZ/N\ZZ)^\times :\langle \ell\rangle]$. The last condition is satisfied for all $\ell$ such that
\begin{itemize}
    \item $\ell \equiv 1\pmod N$
    \item $\ell$ generates $\ZZ/p^2\ZZ$.
\end{itemize}
The density of the set of primes that satisfy these conditions is equal to $\displaystyle\frac{\varphi(p-1)}{\varphi(N)p}$.
\end{remark}

\section{An abelian $p$-covering arising from quotients of $X_l^q(p^nN)$}
Let $\alpha_n$ be the voltage assignment given in Definition~\ref{def:alpha_n}. Let $\beta_n=\det\circ\alpha_n$. Then $\beta_n$ is a voltage assignment with values in $(\ZZ/p^n\ZZ)^\times$. Lemma~\ref{lem:intermediate} tells us that the derived graph $X((\ZZ/p^n\ZZ)^\times,\beta_n)$ is an intermediate covering of $X(\textup{GL}_2(\ZZ/p^n\ZZ),\alpha_n)/X_l^q(N)$. We study this covering in detail in this section.

Recall from Section~\ref{sec:realize} that we fix a basis $\{s_E,t_E\}$ of the Tate module $T_p(E)$ for each $E\in S$. Let 
\[\langle \cdot, \cdot\rangle \colon T_p(E)\times T_p(E)\longrightarrow \Zp(1)\]
be the Weil pairing. In this section, we make the additional assumption that the basis is chosen so that $\langle s_E,t_E \rangle$ is a fixed topological generator $\xi$ of $\Zp(1)$ for all $E\in S$. 
Let $\xi_{n}$ denote the primitive $p^n$-th root of unity given by the image of $\xi$ under the natural map $\Zp(1)\rightarrow \mu_{p^n}$.

\begin{defn}\label{def-zp-graph}We define a directed graph $Y_l^q(p^nN)$ as follows.
    The set of vertices of $Y_l^q(p^nN)$ is given by tuples $(E,R_1,R_2,\zeta)$, where $E,R_1,R_2$ are as before and $\zeta$ is a primitive $p^n$-th root of unity.
Let $\Psi_n$ be the following surjection
\[\Psi_n\colon V(X_l^q(N))\to V(Y_l^q(p^nN)), \quad (E,R_1,R_2,P,Q)\mapsto (E,R_1,R_2,\langle P,Q\rangle).\]
For every vertex $v\in V(Y_l^q(p^nN))$, we fix a pre-image $v'\in V(X_l^q(p^nN))$ under $\Psi_n$. 
 The edges of  $Y_l^q(p^nN)$ are given as follows: 
 for every edge from $v'$ to some $w''\in \Psi_n^{-1}(w)$ in $X_l^q(p^nN)$, we draw an edge from $v$ to $w$ in $Y_l^q(p^nN)$.
\end{defn}

Given $(E,R_1,R_2,\zeta)\in V(Y_l^q(p^nN))$, we may write $\zeta=\xi_n^a$ for a unique $a\in(\ZZ/p^n\ZZ)^\times$. We will often denote $(E,R_1,R_2,\zeta)$ by $(E,R_1,R_2,a)$.

\begin{proposition} \label{prop:iso}
The graph $Y_l^q(p^nN)$ is isomorphic to $X((\ZZ/p^n\ZZ)^\times\beta_n)$. In particular, $Y_l^q(p^nN)$ is the quotient of $X_l^q(p^nN)$ by the action $\textup{SL}_2(\Z/p^n\ZZ)$.
\end{proposition}
\begin{proof}
    By definition, there is a natural bijection
    \[\Theta_n\colon V(Y_l^q(p^nN))\longrightarrow V(X(X_l^q(N),\beta_n)).\] Thus, it remains to show that $\Theta_n$ respects the edges of the two graphs, as in the proof of Theorem~\ref{thm:iso}.
    
    Let $v=(E,R_1,R_2,a)$ and $w=(E',R'_1,R'_2,a')$ be two vertices of $Y_l^q(p^nN)$.  Let $v'$ be the fixed pre-image in $X_l^q(p^nN)$ under $\Psi_n$ as in Definition~\ref{def-zp-graph}. There are $k$ edges from $v$ to $w$ if and only if there are $k$ edges from $v'$ to a vertex $w''$ for some $w''\in \Psi_n^{-1}(w)$ {(note that the vertex $w''$ can depend on the edge $e$)}.
    
    Theorem~\ref{thm:iso} allows us to identify $V(X_l^q(p^nN))$ with $V(X_l^q(N))\times\GL_2(\ZZ/p^n\ZZ)$. Two bases of $E[p^n]$ have the same image in $\mu_{p^n}$ under the Weil pairing if and only if the corresponding matrices in $\GL_2(\ZZ/p^n\ZZ)$ have the same determinant. This implies the existence of $\sigma,\sigma'\in \textup{GL}_2(\ZZ/p^n\ZZ)$ such that $a=\det(\sigma)$, $a'=\det(\sigma')$ and that
    \begin{align*}
        \Psi_n^{-1}(v)&=\left\{(E,R_1,R_2,\sigma\tau)\mid \tau\in \textup{SL}_2(\ZZ/p^n\ZZ)\right\},\\
        \Psi_n^{-1}(w)&=\left\{(E',R'_1,R'_2,\sigma'\tau)\mid \tau\in \textup{SL}_2(\ZZ/p^n\ZZ)\right\}.
    \end{align*}
    
    Without loss of generality, we may assume that the chosen pre-image of $v$ is given by $v'=(E,R_1,R_2,\sigma)$.
    According to the definition of $X(\textup{GL}_2(\ZZ/p^n\ZZ),\alpha_n)$, there is a $1$-to-$1$ correspondence between the edges from $v$ to some $w''\in \Psi_n^{-1}(w)$ and tuples $(e,\tau)$, where $e$ is an edge from $(E,R_1,R_2)$ to $(E',R_1',R_2')$, and $\tau\in \textup{SL}_2(\ZZ/p^n\ZZ)$ is such that
    \[\sigma'\tau=\sigma \alpha_n(e).\]
    The existence of $\tau$ is equivalent to 
    \[a'=\det(\sigma')=\det(\sigma)\det(\alpha_n(e))=a\beta_n(e).\]
    This condition is equivalent to the existence of an edge from $(E,R_1,R_2,a)$ to $(E',R'_1,R'_2,a')$ in $X((\ZZ/p^n\ZZ)^\times,\beta_n)$. Therefore, there is a $1$-to-$1$ correspondence between the tuples $(e,\tau)$ and the edges in  $X((\ZZ/p^n\ZZ)^\times,\beta_n)$ with $v$ as the source. It follows that $\Theta_n$ respects the edges of $Y_l^q(p^nN)$ and $X((\ZZ/p^n\ZZ)^\times,\beta_n)$, as desired. The claim that $Y_l^q(p^nN)$ is a quotient of $X_l^q(p^nN)$ is a direct consequence of the isomorphism $Y_l^q(p^nN)\cong X(X_l^q(N),\beta_n)$ and the definition of voltage graphs.
\end{proof}
\begin{remark}
   Let $\phi:E\rightarrow E'$ be an $\ell$-isogeny and let $P,Q$ be points of order $p^n$ on $E$. The Weil pairing satisfies
   \[\langle \phi(P),\phi(Q)\rangle=\langle P,Q\rangle^\ell.\]
Thus, the graphs $Y_n^\ell(p^nN)$ can be realized as the derived graph of the voltage assignment that assigns $\ell$ to each edge of $X_\ell^q(N)$. We thank one of the two referees for this observation.
\end{remark}
\begin{remark}
    Proposition~\ref{prop:iso} tells us that the graph $Y_l^q(p^nN)$ is independent of the choice of pre-image $v'\in V(X_l^q(p^nN))$ of $v\in V(Y_l^q(p^nN))$ when $N$ is sufficiently large.
\end{remark}

\begin{theorem}
Let $Y_0$ be a connected component of $Y_\ell^q(N)$, Then, the number of connected components of the pre-image of $Y_0$ in $Y_\ell^q(p^nN)$ is uniformly bounded in $n$.
\end{theorem}
\begin{proof} Let $u$ be the minimal non-negative integer such that $l^u\equiv 1\pmod N$. Since $[l^u]$ is  an $l$-power endomorphism on $E$, it follows from Remark~\ref{rem:path} that there is a path from $(E,R_1,R_2,a)$ to $(E,R_1,R_2,l^{2u}a)$ in $Y_l^q(p^nN)$  for all  choices of $(E,R_1,R_2,a)$. Furthermore,  Remark~\ref{rem:connected-comp} tells us that the number of connected components of $Y_l^q(p^nN)$ is bounded by $\vert (\ZZ/p^n\ZZ)^\times /\langle l^{2u}\rangle\vert$, which is uniformly bounded in $n$, as in the proof of Theorem~\ref{thm:number-connected-comp}. Hence, the theorem follows.
\end{proof}

We are now ready to prove Theorem~\ref{thmC}.
\begin{corollary}\label{cor:Zp-tower}  Let $m_0$ be an integer such that the number of connected components of $Y_l^q(p^nN)$, $n\ge m_0$, stabilizes. Let $n\ge m\ge m_0$. Let $Y_m$ be a connected component of $Y_\ell^q(p^mN)$ and let $Y_n$ be its pre-image in $Y_\ell^q(p^nN)$. Then $Y_n/Y_m$ is Galois and $\Gal(Y_n/Y_m)$ is isomorphic to $$\cG_{n,m}:=\left\{x\in(\ZZ/p^n\ZZ)^\times:x\equiv 1\pmod {p^m}\right\}.$$
\end{corollary}
\begin{proof}
    Clearly, $Y_n/Y_m$ is a $\vert \cG_{n,m}\vert$-sheeted covering and $Y_n$ is connected. Via the voltage assignment $\beta_n$, we see that there is an injective group homomorphism
    \[\cG_{n,m}\hookrightarrow \textup{Deck}(Y_n/Y_m).\]
    It suffices to show that this map is surjective. 
    
    Let $\tau\in\Deck(Y_n/Y_m)$ and $v=(E,R_1,R_2,a)\in V(Y_n)$. Let $v'=(E,R_1,R_2,\sigma)\in X_l^q(p^nN)$ be the fixed pre-image under the map $\Psi_n$ given in Definition~\ref{def-zp-graph}. Then there exists  $g\in \cG_{n,m}$ such that $gv=\tau (v)$. Let $g'$ be a lift of $g$ to $\textup{GL}_2(\ZZ/p^n\ZZ)$ (so that $\det(g')=g$). In particular,
    \[\Psi_n(E,R_1,R_2,g'\sigma)=gv.\] 
    Let $\phi$ be an $l$-isogeny on $E$ and let $w=\phi(v)$. Let 
    \[w'=(\phi(E),\phi(R_1),\phi(R_2),\sigma \alpha_n(e_\phi)),\]
    where $e_\phi$ is the edge induced by $\phi$. By definition,
    \[w=\Psi_n(w')=(\phi(E),\phi(R_1),\phi(R_2), a\beta_n(e_\phi)).\] 
    By definition, we see that $\tau(w)=\phi(\tau(v))$. Furthermore,
    \begin{align*}
    \tau(w)&=\phi(\tau(v))=\phi(g v)=\Psi_n(\phi (g' v'))=\Psi_n(\phi(E,R_1,R_2,g'\sigma))\\&=\Psi_n(\phi(E),\phi(R_1),\phi(R_2), g'\sigma \alpha_n(e_\phi))=(\phi(E),\phi(R_1),\phi(R_2), g a \beta_n(e_\phi))\\&=gw.
    \end{align*}
    Thus, $\tau$ is uniquely determined by $\tau(v)$. Hence,  $\cG_{n,m}\cong \textup{Deck}(Y_n/Y_m)$, which concludes the proof.
\end{proof}

In particular, we see that 
\[
Y_{m_0}\leftarrow Y_{m_0+1}\leftarrow Y_{m_0+2}\leftarrow \cdots \leftarrow Y_{m_0+n}\leftarrow \cdots
\]
form an abelian $p$-tower in the sense of \cite[Definition~4.1]{vallieres}.

\begin{example}
    Let $E$ be an elliptic curve such that all $l$-isogenies on $E$ are endomorphims (e.g. when $E$ is supersingular defined over $\mathbb{F}_{13^2}$). Suppose that $N=1$. Let $\phi:E\to E$ be an $l$-isogeny and let $e\in \EE(X_l^q(1))$ be the edge induced by $\phi$.  We have
    \[\langle \phi(s_E),\phi(t_E)\rangle=\langle \widehat{\phi}\circ \phi (s_E),t_E\rangle=\langle s_E,t_E\rangle ^l.\]
    Therefore $\beta_n(e)=l_n$, where $l_n$ denotes the image of $l$ in $(\ZZ/p^n\ZZ)^\times$. In this case, the graph $X_l^q(1)$ has only one vertex with $l+1$ loops (in other words, it is a bouquet). The voltage assignment $\beta_n$ assigns the same value $l_n$ to every edge. Let $Y_0$ be the connected component of $Y_l^q(1)$ that contains $E$. It follows that the tower
    \[Y_{0}\leftarrow Y_{1}\leftarrow Y_{2}\leftarrow \cdots \leftarrow Y_{n}\leftarrow \cdots\]
    is a $\Z_p$-tower with $\lambda$-invariant \begin{align*}\lambda(2(l+1)(2-(T+1)-(T+1)^{-1}))-1=\lambda(2(l+1)T^2)-1=1\end{align*} and $\mu$-invariant \[\mu(2(l+1)(2-(T+1)-(T+1)^{-1}))=\mu(2(l+1)T^2)=\ord_p(2(l+1))\] 
    (see \cite[Definition 2.11]{DLRV}).

 Assume that $p$ does not divide $2(l+1)$ and that $l$ generates $\ZZ/p^2\ZZ$.  Corollary \ref{cor:connected} implies that $X_n^\ss/X^\ss$ is a $\textup{GL}_2(\ZZ/p^n\ZZ)$-covering, and the $\mu$-invaraint of the $\Zp$-tower $(Y_n)_{n\ge0}$ vanishes. Therefore, the $\textup{GL}_2(\Z_p)$-tower
 \[X_0^\ss\leftarrow X_1^\ss\leftarrow X_2^\ss\leftarrow \cdots \leftarrow X_n^\ss\leftarrow \cdots \]
satisfies the so-called $\mathfrak{M}_H(G)$-property and the $K$-theoretic Iwasawa main conjecture holds (see \cite[Sections 7 and 8]{KM} for details).
\end{example}

\section{Isogeny graphs of oriented elliptic curves and Volcano graphs}\label{S:orientation}
Let $M$ be a positive integer coprime to $ql$. While the structure of an $l$-isogeny graph of ordinary elliptic curves equipped with a $\Gamma_1(M)$-level structure is related to the so-called (tectonic) volcano graphs (see \cite[\S4]{LM1}), analogous results for the supersingular counterpart are not available. In \cite{colokohel}, Col\`o--Kohel introduced the concept of orientations of supersingular elliptic curves. The authors in \cite{onuki,arpin-et-all,arpin-win} showed that the isogeny graph of supersingular elliptic curves equipped with an orientation (without any level structures) is a volcano graph, analogous to the ordinary counterpart studied in \cite{kohel}. In this section, we review the notion of orientations and describe the structure of the corresponding isogeny graphs enhanced with a $\Gamma(M)$-level structure for $M>2$. 

Throughout this section, we fix an elliptic curve $E_0$ and assume that one of the following two conditions holds:
\begin{itemize}
\item[(\mylabel{sup}{\textbf{H.ss}})]    $E_0/\FF_{q}$  is supersingular, where $q$ is an even power of a prime number $r$ (so that $\FF_{r^2}\subset\FF_q$) and $r\equiv 1\pmod{12}$.  
    \item[(\mylabel{ord}{\textbf{H.ord}})]  $E_0/\FF_{q}$ is ordinary and $\textup{End}(E_0)\otimes \Q$ does not contain $i=\sqrt{-1}$ and $\zeta_3$, where $\zeta_3$ denotes a primitive $3$-rd root of unity. 
\end{itemize} In both settings, we have $\textup{Aut}(E_0)=\{\pm 1\}$. Throughout, $B$ denotes the quaternion algebra that is ramified only at $r$ and $\infty$.

\subsection{The $l$-isogeny graph of oriented elliptic curves}
We recall the notion of oriented elliptic curves and their isogeny graphs, originally introduced in \cite{colokohel}. We fix once and for all an imaginary quadratic field that admits an embedding
\[
\iota:K\hookrightarrow \End(E_0)\otimes\Q.
\]
If \eqref{sup} holds, $r$ does not split in $K$ and $\iota$ can be realized as an embedding of $K$ into $B$. If \eqref{ord} holds, $\iota$ is an isomorphism of fields.

\begin{defn}\label{def:oriented-elliptic} Let $E/\FF_{q}$ be an elliptic curve that admits an embedding $\iota:K\hookrightarrow\End(E)\otimes \QQ$.
\item[i)] We call the pair $(E,\iota)$ an \textbf{oriented elliptic curve}. 
\item[ii)]     We call an order $\mathcal{O}$ of $K$ \textbf{primitive} for $(E,\iota)$, if $\iota(\mathcal{O})=\iota(K)\bigcap \textup{End}(E)$. 
\item[iii)] We say that $\mathcal{O}$ is $l$-\textbf{primitive} if $l$ does not divide $[\iota(K) \bigcap \textup{End}(E):\iota(\mathcal{O})]$.
\item[iv)] If $\varphi\colon E\to E'$ is an isogeny of elliptic curves, we define the orientation $\varphi_*\iota$ on $E'$ by
\begin{align*}
    \varphi_*\iota:K&\rightarrow\End(E')\otimes \QQ,\\
    \alpha&\mapsto\frac{1}{\deg(\varphi)}{\varphi}\circ\iota(\alpha)\circ\widehat\varphi.
\end{align*}
\item[v)]    An \textbf{isogeny of oriented elliptic curves} $\varphi\colon (E,\iota)\to (E',\iota')$ is an isogeny of elliptic curves $\varphi\colon E'\to E$ such that $\varphi_*\iota=\iota'$.
\item[vi)] We say that $l$-isogenies $\phi,\psi\colon E\to E'$  are \textbf{equivalent isogenies} if they have the same kernel.
\item[vii)] We say that  $(E,\iota)$ and $(E,\iota')$ are \textbf{equivalent  oriented elliptic curves} if there exists an $\overline{\FF_q}$-isomorphism of elliptic curves $\phi\colon E\to E'$ such that $\iota'=\phi_*\iota$. 
\end{defn}

\begin{remark}
    Let $\phi,\psi\colon E\to E'$ be equivalent isogenies. Then $(\phi(E),\phi_*\iota)$ and $(\psi(E),\psi_*\iota)$ are equivalent.
    
In the ordinary case, there are precisely two orientations on $E$, given by the field automorphisms of $K$.  In addition, we can check that $\phi_*\iota=\iota$ for every endomorphism $\phi$ of $E$.
\end{remark}

We are now ready to define the $l$-isogeny graph of oriented elliptic curves.

\begin{defn}
If \eqref{sup} (resp. \eqref{ord}) holds, we define $\cX_l^q$ as the directed graph whose vertices are equivalence classes of oriented elliptic curves $(E,\iota)$ defined over $\FF_q$ that are supersingular (resp. ordinary with $\End(E)\otimes\QQ\cong K$); the edges are given by $l$-isogenies up to equivalence. We define $\widetilde\cX^q_l$ as the undirected graph obtained from $\cX_l^q$ after identifying the edges induced by an isogeny and its dual.
\end{defn}

The following notions will allow us to describe the structure of $\cX_l^q$.
\begin{defn}
Let $\varphi:(E,\iota)\rightarrow(E',\iota')$ be an isogeny of oriented elliptic curves.  Assume that $\mathcal{O}$ and $\cO'$ are orders of $K$ that are primitive for $(E,\iota)$ and  $(E',\iota')$, respectively. We say that $\varphi$ is \begin{itemize}
        \item {\bf horizontal} if $\mathcal{O}=\mathcal{O}'$,
        \item {\bf descending} if $\mathcal{O}'\subset \mathcal{O}$,
        \item {\bf ascending} if $\mathcal{O}'\supset \mathcal{O}$.
    \end{itemize}
   \end{defn}
The notion of horizontal, ascending, and descending isogenies allows us to define the depth of an oriented elliptic curve.
\begin{defn}
    Let $(E,\iota)$ be an oriented elliptic curve. Assume that $\mathcal{O}$ is primitive with respect to $(E,\iota)$. We say that $(E,\iota)$ is of depth zero if the conductor of $\mathcal{O}$ is not divisible by $l$. We say that $(E,\iota)$ is of depth $k$ if there is a path consisting of $k$ descending $l$ -isogenies that starts at a depth-zero oriented elliptic curve $(E',\iota')$ and ends at $(E,\iota)$ . 
\end{defn}

Finally, we introduce the following notions of graphs, which are the last ingredients in our description of $\cX_l^q$.

\begin{defn}\label{def:volcano}
\item[i)] An \textbf{abstract crater} is either a directed cycle graph or a totally disconnected finite graph.
\item[ii)]      An \textbf{infinite volcano graph} $G$ is a directed graph whose vertices may be decomposed as $V(G)=\bigcup_{i=0}^\infty V_i$ such that
\begin{itemize}
    \item The out-degree of every vertex is $l+1$;
    \item The subgraph generated by $V_0$ is  an abstract crater;
    \item  If there is an edge between two vertices $v_i\in V_i$ and $v_j\in V_j$ then $i-j\in \{\pm 1\}$ or $i=j=0$;
    \item Each $v\in V_i$ with $i\ge 1$ has exactly one edge starting at $v$ and ending at a vertex in $V_{i-1}$ and $l$ edges starting at $v$ and ending at vertices in $V_{i+1}$.
\end{itemize} 
\item[iii)]     An \textbf{infinite undirected volcano} $G$ is an undirected $(l+1)$-regular graph whose vertices can be decomposed as $V(G)=\bigcup_{i=0}^\infty V_i$ such that
\begin{itemize}
    \item The subgraph generated by $V_0$, called the \textbf{crater} of $G$, is an undirected cycle graph or totally disconnected;
    \item If there is an edge between $v_i$ and $v_j$ then $i-j\in \{\pm 1\}$ or $i=j=0$;
    \item Each $v\in V_i$ with $i\ge 1$ has exactly one neighbour in $V_{i-1}$ and $l$ neighbours in $V_{i+1}$. 
\end{itemize} 
\item[iv)] A \textbf{finite volcano graph} $G$ is a directed graph whose vertices may be decomposed as $V(G)=\bigcup_{i=0}^kV_i$ such that
\begin{itemize}
    \item The subgraph generated by $V_0$ is  an abstract crater;
    \item  If there is an edge between two vertices $v_i\in V_i$ and $v_j\in V_j$ then $i-j\in \{\pm 1\}$ or $i=j=0$;
    \item Each vertex in $V_{i}$ for $1\le k-1$ has out degree $(l+1)$.
    \item Each $v\in V_i$ with $1\le i\le k-1$ has exactly one edge starting at $v$ and ending at a vertex in $V_{i-1}$ and $l$ edges starting at $v$ and ending at vertices in $V_{i+1}$.
    \item Each vertex in $v\in V_k$ has out-degree $1$. This edge ends in $V_{k-1}$. 
\end{itemize} 
\item[v)] A \textbf{finite undirected volcano graph} $G$ is an undirected graph whose vertices may be decomposed as $V(G)=\bigcup_{i=0}^k V_i$ such that
\begin{itemize}
    \item The subgraph generated by $V_0$, called the \textbf{crater} of $G$, is an undirected cycle graph or totally disconnected ;
    \item If there is an edge between $v_i$ and $v_j$ then $i-j\in \{\pm 1\}$ or $i=j=0$;
    \item Each vertex in $V_i$ for $i\le k-1$ has degree $l+1$.
    \item Each $v\in V_i$ with $1\le i\le k-1$ has exactly one neighbour in $V_{i-1}$ and $l$ neighbours in $V_{i+1}$. 
    \item Each $v\in V_i$ for $i\le k-1$ has degree $(l+1)$.
    \item Each $v\in V_k$ has exactly one neighbour, which lies in $V_{k-1}$.
\end{itemize} 
\end{defn}

\begin{theorem}[{\cite[Section 3.4]{arpin-et-all}}, {\cite[Section 3]{pazuki}}]\label{thm:pazuki}  \label{thm:arpin} If \eqref{sup} holds, the connected components of $\widetilde\cX_l^q$ are infinite undirected volcano graphs whose craters are generated by the depth zero vertices. 
If \eqref{ord} holds, $\widetilde\cX_l^q$ is finite. Its connected components are finite undirected volcanoes whose craters are the depth zero vertices. 
\end{theorem}
\begin{remark}
The description given by Theorem~\ref{thm:arpin} in fact also describes the directed graph $\cX_l^q$ since each undirected edge of $\widetilde\cX_l^q$ represents two edges of $\cX_l^q$ in opposite directions.
\end{remark}

\subsection{The $l$-isogeny graph of oriented elliptic curve with level structure}

The goal of this section is to study oriented elliptic curves enhanced with a $\Gamma(M)$-level structure. In other words, instead of pairs $(E,\iota)$, we consider tuples of the form $(E,\iota, R_1,R_2)$, where $\{R_1,R_2\}$ is a basis for $E[M]$. 
\begin{defn}
\label{defn:isogeny-oriented-level}
    Let $M$ be a positive integer. An isogeny of oriented elliptic curves with a $\Gamma(M)$-level structure $\varphi\colon (E,\iota,R_1,R_2)\to (E',\iota',R'_1,R'_2)$ is an isogeny of oriented elliptic curves $(E,\iota)\to (E',\iota')$ such that $\varphi(R_1)=R'_1$ and $\varphi(R_2)=R'_2$. 
\end{defn}

\begin{defn}\label{def:oriented-graph-level} We fix a set of representatives $S'$ for the equivalence classes of oriented elliptic curves in $V(\cX_l^q)$.
   \item[i)]  Let $\mathcal{X}_l^q(M)$ be the directed graph whose vertices are quadruples $(E,\iota,R_1,R_2)$, where $(E,\iota)\in S'$ and $\{R_1,R_2\}$ is a basis for $E[M]$, and such that the edges are given by $l$-isogenies defined as in Definition~\ref{defn:isogeny-oriented-level}.
    \item[ii)] Let $e\in\EE(\cX_l^q(M))$. We say that $e$ is horizontal/descending/ascending if the isogeny giving rise to $e$ is horizontal/descending/ascending.
\end{defn}

\begin{remark}
\label{rem:level-structures}
    If $\phi:E\rightarrow E'$ is an $l$-isogeny of elliptic curves, the composition $\phi\circ [-1]=[-1]\circ \phi$ is also an $l$-isogeny. Note that $[-1]$  acts non-trivially on a basis of $E[M]$ for $M>2$. Thus, equivalent isogenies induce edges with different targets if $M>2$. For this reason, we have decided to distinguish the edges arising from $\phi$ and $\phi\circ[-1]$ in our definition of $\cX_l^q(M)$ for all positive integers $M$, although these two isogenies are equivalent.  In particular, the graphs $\mathcal{X}_l^q(1)$ and $\mathcal{X}_l^q$ are not isomorphic to each other. They share the same set of vertices, but each edge of $\mathcal{X}_l^q$ corresponds to two distinct edges in $\mathcal{X}_l^q(1)$.

In the ordinary case, we studied similar graphs of $\Gamma_1(M)$-level structure in \cite{LM1}. In contrast to the current setting, the isogenies $\phi\circ[-1]$ and $\phi$ would always give rise to the same edge in the graphs studied in \textit{loc. cit.}
\end{remark}

In order to describe the structure of $\mathcal{X}_l^q(M)$ in more detail, we recall the following notions from \cite{arpin-et-all}.

\begin{defn}\label{def:suitable}
    Let $E/\FF_{q}$ be an elliptic curve that admits an orientation $K\hookrightarrow \End(E)\otimes\QQ$ and let $\theta\in \textup{End}(E)$. Let $f_\theta$ be the minimal polynomial of $\theta$ over $ \QQ$, which we assume to be of degree 2. Let $\alpha\in\CC$ be a root of $f_\theta$ and write $K=\QQ(\alpha)$. Let 
    \[\omega_K=\begin{cases}
        \frac{1+\sqrt{\Delta}}{2} \quad &\Delta\equiv 1 \pmod 4,\\
        \frac{\sqrt{\Delta}}{2} \quad &\Delta\equiv 0\pmod 4,
    \end{cases}\]
    where $\Delta$ is the fundamental discriminant of $K$.
        \item[i)] We say that $\theta$ is $l$-\textbf{primitive} if $\ZZ[\alpha]$ is $l$-primitive with respect to the orientation $K\hookrightarrow \textup{End}(E)\otimes\QQ$ given by $\alpha\mapsto \theta$.
        \item[ii)] We say that $\theta$ is $l$-\textbf{suitable} if $\alpha$ is of the form $f\omega_K+ml$, where $m$ is some integer and $f$ is the conductor of $\ZZ[\alpha]$. 
\end{defn}
\begin{remark}\label{rk:primitive-suitable}
    Given a supersingular elliptic curve $E$ and an orientation $\iota\colon K\to B$, there always exists $\theta \in \textup{End}(E)\bigcap\iota(K)$ that is $l$-primitive and $l$-suitable. Indeed, let $\mathcal{O}$ be an order of $K$ that is primitive with respect to $(E,\iota)$. Let $f$ be the conductor of $\mathcal{O}$ and let $\alpha=f\omega_K$. Then $\mathcal{O}=\ZZ[f\omega_K]=\ZZ[\alpha]$. If we set $\theta=\iota(\alpha)$, the endomorphism $\theta$ is $l$-primitive and $l$-suitable.
    
    In the ordinary case, suppose $\cO$ is the order $\iota^{-1}(\End(E))$ in $K$. Then we can take $\alpha$ as a generator of $\cO$ and set $\theta=\iota(\alpha)$.
\end{remark}

We recall the following results, which are specific to the supersingular setting.

\begin{proposition}\label{prop:eigenvalues}\cite[Proposition 4.8]{arpin-et-all} Assume that $E/\FF_q$ is a supersingular elliptic curve. Let $\theta$ be an element that is $l$-primitive and $l$-suitable for $(E,\iota)$. Assume that $(E,\iota)$ is of depth zero and let $\lambda_1,\lambda_2$ be the eigenvalues of the action of $\theta$ on $E[l]$.
\begin{itemize}
    \item[i)] If $l$ splits in $K$, then $\lambda_1\neq \lambda_2\in \Fl$. Let $P_{\lambda_i}$ be generators of the $\lambda_i$-eigenspaces, then $E\to E/\langle P_{\lambda_i}\rangle$ are horizontal isogenies for $i=1,2$. 
    \item[ii)] If $l$ ramifies in $K$, then $0\neq\lambda_1=\lambda_2\in \Fl$. The $\lambda_i$-eigenspace is one-dimensional, where  $i=1, 2$. Let $P$ be a generator of this eigenspace. Then $E\to E/\langle P\rangle$ is a horizontal isogeny.
    \item[iii)] If $l$ is inert in $K$, then $\lambda_1,\lambda_2\in \mathbb{F}_{l^2}\setminus \Fl$ and there are no horizontal isogenies on $E$. 
\end{itemize}
All other $l$-isogenies are descending in all three cases. 
    
\end{proposition}

\begin{lemma}
\label{lemma:horizonal-isog.}Assume that $l$ splits in $K$. {Let $E/\FF_q$ be a supersingular elliptic curve.}
Let $(E,\iota)$ be a depth-zero vertex.
    Let $\theta\in\End(E)$ be $l$-primitive and $l$-suitable for $(E,\iota)$ and let $\alpha$ be the element attached to $\theta$ given in Definition~\ref{def:suitable}. Let $\lambda_1$ and $\lambda_2$ be given as in Proposition~\ref{prop:eigenvalues}. For $i\in\{1,2\}$, let $P^{E}_i$ be a generator of the $\lambda_i$-eigenspace and $\phi^E_i\colon E\to E/\langle P_i^E\rangle$ be the natural projection.
    \begin{itemize}
        \item[i)]Let $\phi\colon (E,\iota)\to (E',\iota')$ be a horizontal isogeny. Then $\theta':=\iota'(\alpha)$ is $l$-primitive and $l$-suitable for $(E',\iota')$. In particular, $\lambda_1$ and $\lambda_2$ are the eigenvalues of the action of $\theta'$ on $E'[l]$.
        \item[ii)] 
        Let $P_i^{E'}$ be a generator of the $\lambda_i$-eigenspace of $\phi^{E}_i(E)$. Let $i\neq j$. Then $\phi_i^E(P_j)=cP_j^{E'}$, for some $c$ coprime to $l$. 
        \item[iii)] Fix $ i\in\{1, 2\}$. Let $v_0=(E_0,\iota_0,R_{0,1},R_{0,2})=(E,\iota,R_1,R_2)\in V(\cX_l^q(M))$ and define recursively
        \[v_j=(E_j,\iota_j,R_{j,1},R_{j,2})=\phi_i^{E_{j-1}}(v_{j-1}),\ j\ge1. \]
        Then there exists a non-negative integer $r$ such that $v_r=v_0$.
    \end{itemize}
    \end{lemma}
\begin{proof}
i) and ii) follow directly from definitions. It remains to prove iii). By Theorem~\ref{thm:arpin}, there are only finitely many depth-zero vertices in $\cX_l^q$ that lie in the same connected component as $v$. Thus, there is a non-negative integer $r_0$ such that
$E_{r_0}=E$ and $\iota_{r_0}=\iota$.
In particular, there is an $l$-power isogeny $\gamma:(E,\iota)\rightarrow (E,\iota)$ such that $\gamma(R_i)=R_{r_0,i}$ for $i=1, 2$. Let $s$ be the minimal non-negative integer such that $\gamma^s(R_i)=R_i$ for both $i$. Then Part iii) follows in setting $r=r_0s$.
\end{proof}

We now introduce several graph-theoretic notions that will be utilized in our description of $\cX_l^q(M)$.

\begin{defn}\label{def:intertwine}
    Let $Z$ be a directed graph. We define the \textbf{double intertwinement}  $Z^{+-}$ of $Z$ to be the graph such that  $$V(Z^{+-})=\{+v,-v: v\in V(Z)\}$$ (each vertex of $V(Z)$ gives rise to two vertices in $Z^{+-}$) and that $$\EE(Z^{+-})=\{e^{++},e^{+-},e^{-+},e^{--}:e\in\EE(Z)\},$$ where $e^{\bullet\circ}$ denotes an edge whose source and target are $\bullet v$ and $\circ w$, respectively, if $e$ is an edge from $v$ to $w$ in the original graph $Z$. 
\end{defn}
\begin{example} 
Let us illustrate the double intertwinement of the following graph:
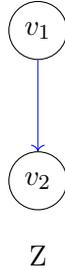
\begin{center}\begin{tikzpicture}[vertex/.style={circle, draw}] 
\node[circle, draw] (a) at (-1,1) {$v_1$};
\node[circle,draw] (b) at (-1,-1) {$v_2$};
\node (e) at (-1,-2) {Z};
\draw[->,blue] (a)--(b);
\end{tikzpicture}
\end{center}
Then $Z^{+-}$ is given by
  
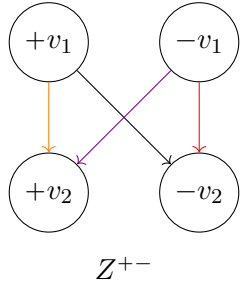
\begin{center}\begin{tikzpicture}[vertex/.style={circle, draw}] 
\node[circle, draw] (a) at (-1,1) {$+v_1$};
\node[circle,draw] (b) at (-1,-1) {$+v_2$};
\node[circle, draw] (c) at (1,1) {$-v_1$};
\node[circle,draw] (d) at (1,-1) {$-v_2$};
\node (e) at (0,-2) {$Z^{+-}$};
\draw[->,orange] (a)--(b);
\draw[->,red] (c)--(d);
\draw[->] (a)--(d);
\draw[->,violet] (c)--(b);
\end{tikzpicture}
\end{center}

The double intertwinement of the directed cycle graph with $4$ edges and $4$ vertices is given by:
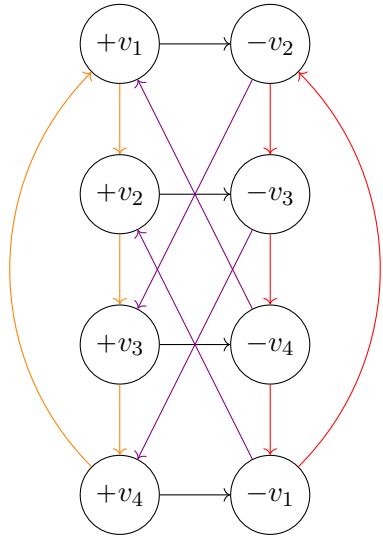
\begin{center}
\label{ex:double-int}
    \begin{tikzpicture}[vertex/.style={circle,draw}]
        \node[circle,draw] (a) at (-1,2) {$+v_1$};
        \node[circle,draw] (b) at (-1,0) {$+v_2$};
        \node[circle,draw] (c) at (-1,-2) {$+v_3$};
        \node[circle,draw] (d) at (-1,-4) {$+v_4$};
        \node[circle,draw] (e) at (1,2) {$-v_2$};
        \node[circle,draw] (f) at (1,0) {$-v_3$};
        \node[circle,draw] (g) at (1,-2) {$-v_4$};
        \node[circle,draw] (h) at (1,-4) {$-v_1$};
        \draw[->,orange] (a)--(b);
        \draw[->,orange] (b)--(c);
        \draw[->,orange] (c)--(d);
        \draw[->,orange] (d)to[in=225,out=135](a); 
        \draw[->,red] (e)--(f);
        \draw[->,red] (f)--(g);
        \draw[->,red] (g)--(h);
        \draw[->,red] (h)to[in=-45,out=45](e);
        \draw[->] (a)--(e);
        \draw[->] (b)--(f);
        \draw[->] (c)--(g);
        \draw[->] (d)--(h);
        \draw[->,violet] (g)--(a);
        \draw[->, violet] (e)--(c);
        \draw[->,violet] (h)--(b);
        \draw[->,violet] (f)--(d);
    \end{tikzpicture}
\end{center}
\end{example}
The following definition was introduced in \cite[\S5]{LM1}.

\begin{defn}
  Let $\r,\s,\t,\cc$ be non-negative integers. We say that a directed graph is an \textbf{abstract tectonic crater} of parameters $(\r,\s,\t,\cc)$ if it satisfies
    \begin{itemize}
  \item[a)] There are $\r\s\t$ vertices;
  \item[b)] Each edge is assigned a color -- blue or green;
  \item[c)] At each vertex $v$, there is exactly one blue edge with $v$ as the source, and exactly one blue edge with $v$ as the target, and similarly for green edges;
  \item[d)] Starting at each vertex, there is exactly one closed blue (resp. green) path of length $\r\s$ (resp. $\r\t$);
  \item[e)] After every $\s$ (resp. $\cc\t$) steps in the closed blue (resp. green) paths given in d), the two paths meet at a common vertex.  
  \end{itemize}
\end{defn}

\begin{remark}
Let $v$ be a vertex in an abstract tectonic crater of parameters $(\r,\s,\t,\cc)$. Let $C_b$ and $C_g$ denote the closed blue and green paths passing through $v$ given by d). By e), $C_b$ can be decomposed into $\r$ paths of length $\s$, each of which starts and ends at the vertices of $C_b\bigcap C_g$. Along $C_g$, these $\r$ vertices are reached after $n\cc \t$ steps, where $n=1,2,\ldots$ In particular, the cyclic group generated by $\cc\t$ in $\ZZ/\r\t\ZZ$ should have order $\r$. This means that $\cc$ and $\r$ should be coprime. Furthermore, we see that $C_g$ can be decomposed into $\r$ paths of length $\t$, each of which starts and ends at vertices of $C_b\bigcap C_g$.

Let $v'$ be a vertex of $C_b\bigcap C_g$ such that there is a path of length $\t$ in $C_g$ going from $v$ to $v'$. Our discussion above implies that there exists a positive integer $\cc'$ such that there is a path of length $\cc'\s$ in $C_b$ going from $v$ to $v'$. Therefore, we may reverse the roles of $\s$ and $\t$ in e), replacing $\s$ and $\cc\t$ by $\cc'\s$ and $\t$, respectively.
\end{remark}

\begin{example} An abstract tectonic crater with $\t=\s=1$, $\cc=2$ and $\r=5$.
    
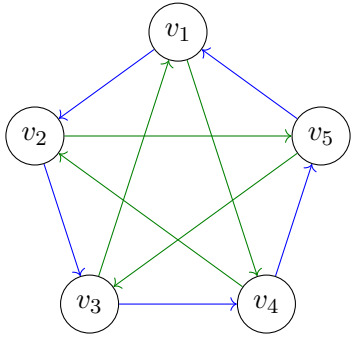
\begin{center}
           \begin{tikzpicture}[vertex/.style={circle, draw}]
\foreach \X[count=\Y] in {blue,red,green,yellow,green}
{\node[draw,circle,black] (x-\Y) at ({72*\Y+18}:2){$v_\Y$}; }
\foreach \X[count=\Y] in {0,...,4}
{\ifnum\X=0
\draw[line width=0.5mm,->,blue] (x-5) --(x-\Y) ;
\else
\draw[line width=0.5mm,->,blue] (x-\X) --(x-\Y) ;
\fi}
\draw[dashed,->,Green] (x-1)--(x-4);
\draw[dashed,->,Green] (x-4)--(x-2);
\draw[dashed,->,Green] (x-2)--(x-5);
\draw[dashed,->,Green] (x-5)--(x-3);
\draw[dashed,->,Green] (x-3)--(x-1); 
     \end{tikzpicture}
    \end{center}
\end{example}
For further examples, the reader is referred to \cite[\S4.2]{LM1}.
\begin{remark}
    Note that the double intertwinment of the directed cycle graph with $4$ vertices in Example~\ref{ex:double-int} is a tectonic crater with parameters $(2,2,2,1)$. More generally, the double intertwinment of a directed cycle graph $C_n$ of length $n$ is a tectonic crater with parameters $(n/2,2,2,1)$. The green edges are given by $\{e^{++},e^{--}:e\in\EE(C_n)\}$ while the blue edges are given by $\{e^{+-},e^{-+}:e\in\EE(C_n)\}$.
\end{remark}
We may enhance the definition of an abstract tectonic crater, giving the "tectonic" version of a volcano graph:

  \begin{defn}\label{def:crater}
 \item[i)]  An \textbf{infinite tectonic volcano} $G$ is a directed graph whose vertices may be decomposed into $V(G)=\bigcup_{i=0}^\infty V_i$ such that
  \begin{itemize}
      \item The out-degree of every vertex is $l+1$;
      \item  The subgraph generated by $V_0$ is an abstract tectonic crater (we shall refer to $V_0$ as the tectonic crater of $G$);
      \item If there is an edge between two vertices $v_i\in V_i$ and $v_j\in V_j$, then $i-j\in \{\pm 1\}$ or $i=j=0$;
      \item Every $v\in V_i$ with $i\ge 1$ has exactly one edge starting at $v$ and ending at a vertex in $V_{i-1}$ and $l$ edges starting at $v$ and ending at a vertex in $V_{i+1}$.
  \end{itemize}
  \item[ii)]  A \textbf{finite tectonic volcano} $G$ is a directed graph whose vertices may be decomposed into $V(G)=\bigcup_{i=0}^k V_i$ such that
  \begin{itemize}
      \item The out-degree of every vertex $v\in \bigcup_{i=0}^{k-1}V_i$ is $l+1$;
      \item  The subgraph generated by $V_0$ is an abstract tectonic crater (we shall refer to $V_0$ as the tectonic crater of $G$);
      \item If there is an edge between two vertices $v_i\in V_i$ and $v_j\in V_j$, then $i-j\in \{\pm 1\}$ or $i=j=0$;
      \item Every $v\in V_i$ with $1\le i\le k-1$ has exactly one edge starting at $v$ and ending at a vertex in $V_{i-1}$ and $l$ edges starting at $v$ and ending at a vertex in $V_{i+1}$.
      \item Every vertex in $v\in V_k$ has out degree $1$ and this edge ends in $V_{k-1}$.
  \end{itemize}
\end{defn}

We are now ready to prove our first theorem on the structure of $\cX_l^q(M)$.
\begin{theorem}\label{thm:volcano}
    Assume that $l$ splits in $K$ and $M>2$. Let $\cY_M$ be a connected component of $\mathcal{X}_l^q(M)$. Then $\cY_M$ is the double intertwinement of an infinite tectonic (resp. a finite tectonic) volcano graph {if \eqref{sup} (resp. \eqref{ord}) holds.} 
\end{theorem}
\begin{proof}
    There is a natural map $\pi\colon \cY_M\to \cX_l^q$ given by \[(E,\iota, R_1,R_2)\mapsto (E,\iota).\] Note that this is not a covering of graphs. {For example, if there are vertices whose depth are positive, the depth-zero} vertices of $\cY_M$ have out-degree $2(l+1)$, while those in $\cX_l^q$ have out-degree $(l+1)$. 
    
    It follows from Theorem~\ref{thm:arpin} that $\cX_l^q$ is the union of (finite) volcanoes. 
    Let $(E,\iota)$ be a depth-zero vertex of $\cX_l^q$ and let $v$ be a lift of $(E,\iota)$ in $\cY_M$. For each pair $(E_1,E_2)$ of elliptic curves, we fix a set of representatives $S(E_1,E_2)$ of $\{\textup{$l$-isogenies from $E_1$ to $E_2$}\}/\{\pm 1\}$. 

    Let $\theta\in \textup{End}(E)$ be $l$-primitive and $l$-suitable for $(E,\iota)$ and let $\alpha=\iota^{-1}(\theta)$ as in Remark~\ref{rk:primitive-suitable}.  It follows from Proposition~\ref{prop:eigenvalues}i)  that in the supersingular case, the eigenvalues $\lambda_1$ and $\lambda_2$ of the action of $\theta$ on $E[l]$ are distinct elements of $\Fl$; {the same holds in the ordinary case since $E[l]\cong \cO/\fL_1\oplus\cO/\fL_2$, where $\cO=\iota^{-1}(\End(E))$, and $\fL_1$ and $\fL_2$ are the ideals of $\cO$ lying above $l$}. Let $S_0(E_1,E_2)\subset S(E_1,E_2)$ be the subset of horizontal isogenies. Starting at $v$ and propagating only along the edges given by $S_0(E_1,E_2)$, we obtain a finite connected subgraph $Z_0$ of $\cY_M$. We will show that $Z_0$ is an abstract tectonic crater. 

    Let $(E',\iota',R'_1,R'_2)$ be an arbitrary vertex in $Z_0$. We define $\theta'=\iota'(\alpha)$.  The eigenvalues of the action of $\theta'$ on $E'[l]$ are once again $\lambda_1$ and $\lambda_2$; {this follows from Lemma~\ref{lemma:horizonal-isog.}i) in the supersingular case, and the decomposition $E'[l]\cong \cO/\fL_1\oplus\cO/\fL_2$ in the ordinary case (as $E$ and $E'$ are related via horizontal isogenies). In other words, the eigenvalues on $l$-torsions are invariant within the vertices of $Z_0$}. For $1\le i \le 2$, let $P^{E'}_i$ be a $\lambda_i$-eigenvector of the action of $\theta'$ on $E[l]$. Let \[\phi^{E'}_i\colon E'\to E'/\langle P^{E'}_i\rangle\] be the natural projection. We say that an edge in $Z_0$ is blue (resp. green) if it is induced by $\phi^{E'}_1$ or by $\phi^{E'}_1\circ[-1]$ (resp. by $\phi^{E'}_2$ or $\phi^{E'}_2\circ [-1]$). By the definition of $S_0(E_1,E_2)$, exactly one of the two elements in $\{\phi^{E'}_i,\phi^{E'}_i\circ [-1]\}$ induces an edge in $Z_0$. {Note that blue and green edges commute in the following sense. Let $v\in V(Z_0)$, and suppose that we obtain $v_g$ (resp. $v_b$) after applying the unique green (resp. blue) edge with $v$ as the source. Then, the target of the unique blue edge with $v_g$ as the source coincides with the target of the unique green edge with $v_b$ as the source. In the supersingular case, this follows from Lemma~\ref{lemma:horizonal-isog.}ii); in the ordinary case, it follows from the fact that $\mathcal{O}$ is commutative.}
    
    For every vertex $v'$ in $Z_0$, we denote by $\phi_{i,j}(v')$ the vertex at the end of a path consisting of $i$ blue and $j$ green edges that have $v'$ as the source. Using a similar proof as the one given in \cite[Lemma 4.4]{LM1}, we can show that the following holds.
    \begin{equation}
    \label{analogon-4-4}
        \textup{If $\phi_{i,j}(v)=\phi_{i',j'}(v)$, then $\phi_{i,j}(v')=\phi_{i',j'}(v')$ for all $v'\in V(Z_0)$.}
    \end{equation}
    In the supersingular case, Lemma~\ref{lemma:horizonal-isog.}iii) tells us that there exists a non-negative integer $h_1$ (resp. $h_2$) such that a blue (resp. green) path of length $h_1$ (resp. $h_2$) sends $v$ to itself. {In the ordinary case, the same holds if we take $h_i$ such that $\mathfrak{L}_i^{h_i}=(a)$, with $a\equiv 1\pmod M$.} Note that $\phi_{h_1,0}(v)=\phi_{0,h_2}(v)=v$. Let $s$ be the minimal integer such that $0<s<h_1$ such that there exists an integer $j$ satisfying $\phi_{s,0}(v)=\phi_{0,j}(v)$. Let $t=\gcd(h_2,j)$ and let $c t=j$. Let  $(s',t',c')$ be a triple of integers satisfying the following properties 
    \begin{itemize}
        \item $s'\mid h_1$.
        \item $t'\mid h_2$.
        \item $c'$ is coprime to $h_2/t'$. 
        \item $\phi_{s',0}(v)=\phi_{0,c't'}(v)$.
    \end{itemize}
    Using \eqref{analogon-4-4} instead of \cite[Proposition 4.4]{LM1} in the proof of \cite[proposition 4.6]{LM1}, we deduce:
    \begin{equation}
        \label{analogon-4-6}
        \textup{There exists an integer $d$ such that $s'=ds$ and $t'=dt$.}
    \end{equation}     
    Substituting Proposition 4.4 in \textit{loc. cit.} by \eqref{analogon-4-4}, and Proposition 4.6 by \eqref{analogon-4-6}, it follows from the same line of argument as given in \cite[Section 4.2]{LM1} that $Z_0$ is indeed a tectonic volcano.
    
    Let $Z$ be the subgraph obtained by starting at $v$ and propagating only along the edges in $S(E_1,E_2)$. Then $Z_0$ is a subgraph of $Z$ and all horizontal edges of $Z$ lying in $Z_0$. We next show that $Z$ is a tectonic volcano graph, with $Z_0$ as the crater. We consider the supersingular and ordinary cases separately.
    
    Suppose that $E$ is supersingular. By Proposition~\ref{prop:eigenvalues}, every vertex in $Z_0$ is the source of exactly $l-1$ descending edges in $Z$. By \cite[Proposition~4.1]{onuki}, every vertex in $V(Z)\setminus V(Z_0)$ is the source of one ascending edge and $(l-1)$ descending edges. Therefore, $Z$ is indeed a tectonic volcano with $Z_0$ as its tectonic crater.     
    If $E$ is ordinary, the number of descending and ascending vertices is given by Theorem~\ref{thm:pazuki}. In particular, $Z$ is a finite tectonic volcano with $Z_0$ as its crater.
    
    The last step of our proof constitutes of showing that $\cY_M$ is the double intertwinement of $Z$.
    For every vertex $v=(E,\iota, R_1,R_2)$ in $Z$, the definition of $S(E_1,E_2)$ implies that the "opposite" vertex $v':=(E,\iota, -R_1,-R_2)$ does not lie in $Z$. We may find a path from $v$ to $v'$ by composing an appropriate power of $[l]$ with $[-1]$. Thus, $v'\in V(\cY_M)\setminus V(Z)$. In particular,  we may identify $v$ and $v'$ with $+v$ and $-v$ in $ V(Z^{+-})$, respectively. 
    
    If there is an edge $e$ from $v$ to $w$ in $Z$, there is one from $v$ to $-w$ in $\cY_M$ simply by composing the path from $v$ to $w$ with $[-1]$. Similarly, there are edges from $-v$ to $-w$ and from $-v$ to $w$. Therefore, $e\in \EE(Z)$ gives rise to four edges in $\cY_M$, which we may identify with the edges $e^{\bullet\circ}$, $\bullet,\circ\in\{+,-\}$ in $Z^{+-}$.

    As all edges in $\cY_M$ are given by compositions of the edges in the set $S(E_1,E_2)$ and $[-1]$, we conclude that $\cY_M=Z^{+-}$.
\end{proof}
In a similar manner, one can prove:
\begin{theorem}\label{thm:volcano2}
    Assume that $l$ ramifies in $K$ and $M>2$. {Furthermore, assume that  \eqref{sup} (resp. \eqref{ord} holds)}. Then $\cY_M$ is the double intertwinement of an infinite (resp. finite) volcano graph with a connected crater.
\end{theorem}

\begin{theorem}\label{thm:volcano3}
    Assume that $l$ is inert in $K$, $M>2$. {Furthermore, assume that \eqref{sup}  (resp. \eqref{ord}) holds}.  Then $\cY_M$ is the double intertwinement of an infinite (resp. finite) volcano graph whose crater is disconnected.
\end{theorem}
\begin{remark}
The graph $\cY_M$ studied in Theorems \ref{thm:volcano}, 
 \ref{thm:volcano2} and \ref{thm:volcano3} is finite if and only if \eqref{ord} holds. In this case, the proofs are similar to the corresponding results in \cite[\S4]{LM1}. The main divergence is the identification of the edges arising from $\phi$ and $\phi\circ[-1]$ (see Remark~\ref{rem:level-structures}).

    If $M=2$, the vertices $(E,\iota,R_1,R_2)$ and $(E,\iota,-R_1,-R_2)$ coincide. Furthermore, $\phi$ and $\phi\circ [-1]$ define the same edge. Thus, instead of the double intertwinement of a (tectonic) volcano, the graph $\mathcal{Y}_2$ is derived from a (tectonic) volcano by drawing every edge with multiplicity $2$. 
\end{remark}

\begin{defn}
    We define $X_l^q(M)^{\ss}$ as the subgraph of $X_l^q(M)$ (c.f. Defintiion~\ref{def:intro}) generated by the vertices arising from supersingular elliptic curves. 
\end{defn}
\begin{corollary}\label{cor:quotient-volcano} 
    Suppose that $M>2$ and $r\equiv 1\pmod{12}$.
    Then each connected component of $X_l^q(M)^{\ss}$ is the quotient of the double intertwinement of an infinite volcano graph or an infinite tectonic volcano graph. 
\end{corollary}

\begin{proof}
    This follows from the fact that each connected component of $X_l^q(M)^{\ss}$ is the quotient of certain $\cY_M$ given in  Theorems~\ref{thm:volcano}, \ref{thm:volcano2} and \ref{thm:volcano3}, where the covering map is given by sending the vertex $(E,\iota, R_1,R_2)$ to $(E,R_1,R_2)$.
\end{proof}\begin{example}
    Put $r=13$. Then there exists (up to isomorphism) exactly one supersingular elliptic curve $E$ defined over $\mathbb{F}_{13^2}$. Let $K=\Q(\sqrt{-2})$. Then $K$ is a CM field of $E$. If $l$ is inert in $K$, the curve $E$ does not admit a horizontal $l$-isogeny. In particular, there are no horizontal edges in the graph $X_l^q(M)^\ss$. Every connected component of $X_l^q(M)^{\ss}$ has a totally disconnected crater. If $l=2$, then $l$ is ramified in $K$. In this case, each vertex of depth zero in $X_l^q(M)^{\ss}$ admits exactly $2$ horizontal isogenies,  say $\phi_1$ and $\phi_2=-\phi_1$. Then the crater is the double interwinment of a cycle graph. If $l$ splits in $K$, Theorem~\ref{thm:volcano} says that the crater is the double intertwinment of a tectonic crater. 
    \end{example}
\subsection{Towers of isogeny graphs of oriented elliptic curves}
We prove the following analogue of Theorem~\ref{thm:ordinary-components} in the context of oriented supersingular elliptic curves.
\begin{theorem} \label{thm:oriented-components}
Suppose that \eqref{sup} holds.
    Let $\cX_0$ be a connected component of $\cX_l^q(N)$. For each integer $n\ge1$, let $\cX_n$ be the pre-image of $\cX_0$ in $\cX_l^q(p^nN)$. For $n$ sufficiently large, the number of connected components in $\cX_n$ is given by $cp^{2(n-1)}$ (resp. $cp^{3(n-1)}$) if $p$ is split (resp. non-split) in $K$, where $c\ge (p^2-p)$ is a constant independent of $n$.
\end{theorem}
\begin{proof}
Let $v=(E,\iota,R_1,R_2)$ be an element of the crater of $\cX_0$. Let $\cO$ be an order of $K$ that is primitive for $(E,\iota)$.  Since $v$ belongs to the crater of $\cX_0$, we have $l\nmid [\cO_K:\cO]$. In particular, the splitting behavior of $l$ in $K$ is the same as that in $\cO$.

Consider the natural graph covering $\cX_n\rightarrow \cX_0$. 
 Let $(E,\iota,R_1,R_2,P,Q)$ and $(E,\iota,R_1,R_2,P',Q')$ be two pre-images of $v$ in $V(\cX_n)$. By Remark~\ref{rem:path}, they lie in the same connected component of $\cX_n$ if and only if there is an $l$-power isogeny $\phi\colon E\to E$ such that $\phi_*\iota=\iota$, $\phi(R_1)=R_1$, $\phi(R_2)=R_2$, $\phi(P)=P'$ and $\phi(Q)=Q'$. Let $U_0$ denote the set of such isogenies.

The condition that $\phi_*\iota=\iota$ can be rewritten as
\[
\phi\circ\hat\phi\circ\iota(\alpha)=\phi\circ\iota(\alpha)\circ\hat{\phi},\ \forall\alpha\in K.
\]
Thus, as elements of the quaternion algebra $B$, $\hat{\phi}$ (and hence $\phi$) belongs to the center of $\iota(K)$ in $B$, which is $\iota(K)$ itself by general properties of quaternion algebras. Therefore, $\phi$ belongs to $\iota(\cO)\subset \End(E)$. Conversely, it is clear that any $\phi\in\iota(\cO)$ satisfies $\iota_*\iota=\iota$. Hence, we can identify the set $U_0$ with the multiplicative set of elements in $\mathcal{O}$ of $l$-power norm that act trivially on $E[N]$. We denote this set by $U$.

It is clear that the kernel of the action of $\cO$ on $E[p^n]$ is $p^n\cO$. Therefore, if $\phi(P)=\phi'(P)$ and $\phi(Q)=\phi'(Q)$ if and only if the images of $\phi$ and $\phi'$ in $U$ are congruent modulo $p^n$.
Thus, the number of elements in a connected component of $V(\cX_n)$ is equal to $|U_n|$, where $U_n$ is the image of $U$ in $(\cO/p^n)^\times$.
In addition, the degree of the covering $\cX_n/\cX_0$ is $|\GL_2(\ZZ/p^n\ZZ)|$. Therefore, as in the proof of Theorem~\ref{thm:ordinary-components}, the number of connected components of $\cX_n$ is given by $|\GL_2(\ZZ/p^n\ZZ)|/|U_n|$. We can now proceed as before to conclude the proof of the theorem.
\end{proof}

As in Corollary~\ref{cor:ord-not-Galois}, we have the following corollary, which proves part (i) of Theorem~\ref{thmE}.
\begin{corollary} \label{cor:ss-not-Galois}
    The covering $\cX_{n}/\cX$ is not Galois for $n$ sufficiently large.
\end{corollary}

Finally, the proof of Theorem~\ref{thm:abelian} is readily extendable to this setting, which proves part (ii) of Theorem~\ref{thmE}. 

\begin{theorem}\label{thm:ss-abelian}
Suppose that \eqref{sup} holds. Let $\cX$ be a connected component of $\cX_l^q(N)$. Let $\cX_0'=\cX$. For each integer $n\ge1$, let $\cX_n'$ be a connected component of the pre-image of $\cX_{n-1}'$ in $\cX_l^q(p^nN)$. Then $\cX_n'/\cX$ is a Galois covering. Furthermore, $\displaystyle\varprojlim_n\Gal(\cX_n'/\cX)$ is an abelian $p$-adic Lie group of dimension 2 (resp. dimension 1) if $l$ is split (resp. non-split) in $K$.
\end{theorem}

\begin{remark}
    In the case where \eqref{ord} holds, each connected component of $\cX_l^q(p^nN)$ corresponds to one of the two field automorphisms of $K$. Therefore, the analogue of Theorem~\ref{thmE} in the setting of oriented ordinary elliptic curves is already covered by Theorem~\ref{thmA}.
\end{remark}

\subsection*{Data availability statement}
Data sharing not applicable to this article as no datasets were generated or analyzed during the current study.
\subsection*{Conflict of interest statement}All authors have no conflicts of interest.

\bibliographystyle{amsalpha}
\bibliography{references}

\newcommand{\etalchar}[1]{$^{#1}$}
\providecommand{\bysame}{\leavevmode\hbox to3em{\hrulefill}\thinspace}
\providecommand{\MR}{\relax\ifhmode\unskip\space\fi MR }
\providecommand{\MRhref}[2]{%
  \href{http://www.ams.org/mathscinet-getitem?mr=#1}{#2}
}
\providecommand{\href}[2]{#2}
\begin{thebibliography}{CFK{\etalchar{+}}05}

\bibitem[ACL{\etalchar{+}}23]{arpin-et-all}
Sarah Arpin, Mingjie Chen, Kristin~E. Lauter, Renate Scheidler, Katherine~E.
  Stange, and Ha~T.~N. Tran, \emph{Orienteering with one endomorphism},
  Matematica \textbf{2} (2023), no.~3, 523--582.

\bibitem[ACL{\etalchar{+}}24]{arpin-win}
\bysame, \emph{Orientations and cycles in supersingular isogeny graphs},
  Research Directions in Number Theory: Women in Numbers V (2024), 25--86.

\bibitem[Arp22]{arpin}
Sarah Arpin, \emph{Adding level structure to supersingular elliptic curve
  isogeny graphs}, 2022, to appear in J. Théor. Nombres Bordeaux, available at
  arXiv:2203.03531.

\bibitem[BCP22]{pazuki}
Henry Bambury, Francesco Campagna, and Fabien Pazuki, \emph{{Ordinary isogeny
  graphs over $\mathbb{F}_p$: The inverse volcano problem}}, 2022, to appear in
  Annali Scuola Norm Sup Pisa, available at arxiv: 2210.01086.

\bibitem[CFK{\etalchar{+}}05]{CFKSV}
John Coates, Takako Fukaya, Kazuya Kato, Ramdorai Sujatha, and Otmar Venjakob,
  \emph{The {$\rm GL_2$} main conjecture for elliptic curves without complex
  multiplication}, Publ. Math. Inst. Hautes \'{E}tudes Sci. (2005), no.~101,
  163--208.

\bibitem[CK20]{colokohel}
Leonardo Col\`o and David Kohel, \emph{Orienting supersingular isogeny graphs},
  J. Math. Cryptol. \textbf{14} (2020), no.~1, 414--437.

\bibitem[CL23]{codogni-lido}
Giulio Codogni and Guido Lido, \emph{Spectral theory of isogeny graphs}, 2023,
  preprint, arXiv:2308.13913.

\bibitem[DLRV24]{DLRV}
Cédric Dion, Antonio Lei, Anwesh Ray, and Daniel Vallières, \emph{On the
  distribution of {I}wasawa invariants associated to multigraphs}, Nagoya Math.
  J. \textbf{253} (2024), 48--90.

\bibitem[ECIR11]{interlando-elia-ii}
Michele Elia, J.~Carmelo~Interlando, and Ryan Rosenbaum, \emph{On the structure
  of residue rings of prime ideals in algebraic number fields---{P}art {II}:
  ramified primes}, Int. Math. Forum \textbf{6} (2011), no.~9-12, 565--589.

\bibitem[EIR10]{interlando-elia-i}
Michele Elia, J.~Carmelo Interlando, and Ryan Rosenbaum, \emph{On the structure
  of residue rings of prime ideals in algebraic number fields {P}art {I}:
  unramified primes}, Int. Math. Forum \textbf{5} (2010), no.~53-56,
  2795--2808.

\bibitem[GK21]{gorenkassaei}
Eyal~Z. Goren and Payman~L. Kassaei, \emph{{$p$}-adic dynamics of {H}ecke
  operators on modular curves}, J. Th\'{e}or. Nombres Bordeaux \textbf{33}
  (2021), no.~2, 387--431.

\bibitem[Gon21]{gonet-thesis}
Sophia~R. Gonet, \emph{Jacobians of {F}inite and {I}nfinite {V}oltage {C}overs
  of {G}raphs}, ProQuest LLC, Ann Arbor, MI, 2021, Thesis (Ph.D.)--The
  University of Vermont and State Agricultural College.

\bibitem[Gon22]{gonet22}
\bysame, \emph{Iwasawa theory of {J}acobians of graphs}, Algebr. Comb.
  \textbf{5} (2022), no.~5, 827--848.

\bibitem[Iwa69]{iwasawa69}
Kenkichi Iwasawa, \emph{Analogies between number fields and function fields},
  Some {R}ecent {A}dvances in the {B}asic {S}ciences, {V}ol. 2 ({P}roc.
  {A}nnual {S}ci. {C}onf., {B}elfer {G}rad. {S}chool {S}ci., {Y}eshiva {U}niv.,
  {N}ew {Y}ork, 1965-1966), Yeshiva Univ., Belfer Graduate School of Science,
  New York, 1969, pp.~203--208.

\bibitem[Iwa73]{iwasawa73}
\bysame, \emph{On {${\bf Z}_{l}$}-extensions of algebraic number fields}, Ann.
  of Math. (2) \textbf{98} (1973), 246--326.

\bibitem[KM23]{KM}
Sören Kleine and Katharina Müller, \emph{On the non-commutative {I}wasawa
  main conjecture for voltage covers of graphs}, 2023, to appear in Israel J.
  Math., available at arXiv:2307.15395.

\bibitem[Koh96]{kohel}
David~Russell Kohel, \emph{Endomorphism rings of elliptic curves over finite
  fields}, ProQuest LLC, Ann Arbor, MI, 1996, Thesis (Ph.D.)--University of
  California, Berkeley.

\bibitem[LM24a]{LM1}
Antonio Lei and Katharina Müller, \emph{On ordinary isogeny graphs with level
  structure}, Expo. Math. \textbf{42} (2024), no.~5, 125589.

\bibitem[LM24b]{LM-zeta}
\bysame, \emph{On the zeta functions of supersingular isogeny graphs and
  modular curves}, Arch. Math. \textbf{122} (2024), 285--294.

\bibitem[LV23]{leivallieres}
Antonio Lei and Daniel Valli\`eres, \emph{The non-{$\ell$}-part of the number
  of spanning trees in abelian {$\ell$}-towers of multigraphs}, Res. Number
  Theory \textbf{9} (2023), no.~1, Paper No. 18, 16.

\bibitem[MV23]{vallieres2}
Kevin McGown and Daniel Valli\`eres, \emph{On abelian {$\ell$}-towers of
  multigraphs {II}}, Ann. Math. Qu\'{e}. \textbf{47} (2023), no.~2, 461--473.

\bibitem[MV24]{vallieres3}
\bysame, \emph{On abelian {$\ell$}-towers of multigraphs {III}}, Ann. Math.
  Qu\'{e}. \textbf{48} (2024), no.~1, 1--19.

\bibitem[Onu21]{onuki}
Hiroshi Onuki, \emph{On oriented supersingular elliptic curves}, Finite Fields
  Appl. \textbf{69} (2021), Paper No. 101777.

\bibitem[Rod19]{thesis-roda}
Megan Roda, \emph{Supersingular isogeny graphs with level $n$ structure and
  path problems on ordinary isogeny graphs}, 2019, Master thesis, McGill
  University,
  \href{https://escholarship.mcgill.ca/concern/theses/c247dx821}{https://escholarship.mcgill.ca/concern/theses/c247dx821}.

\bibitem[Sug17]{sugiyama}
Kennichi Sugiyama, \emph{Zeta functions of {R}amanujan graphs and modular
  forms}, Comment. Math. Univ. St. Pauli \textbf{66} (2017), no.~1-2, 29--43.

\bibitem[Val21]{vallieres}
Daniel Valli\`eres, \emph{On abelian {$\ell$}-towers of multigraphs}, Ann.
  Math. Qu\'{e}. \textbf{45} (2021), no.~2, 433--452.

\end{thebibliography}
\end{document}